\documentclass[11pt]{amsart}

\usepackage{amsmath}
\usepackage{amsfonts}
\usepackage{amssymb}
\usepackage{amscd}
\usepackage{amsthm}
\usepackage{framed}
\usepackage{fullpage}
\usepackage{graphicx}
\usepackage{latexsym}
\usepackage[numbers]{natbib}

\usepackage{hyperref}
\hypersetup{colorlinks,linkcolor={red},citecolor={blue},urlcolor={blue}}

\evensidemargin\oddsidemargin

\newtheorem{theorem}{Theorem}[section]
\newtheorem{lemma}[theorem]{Lemma}

\newtheorem{proposition}[theorem]{Proposition}

\theoremstyle{definition}
\newtheorem{definition}[theorem]{Definition}

\newtheorem{question}[theorem]{Question}

\numberwithin{equation}{section}

\newcommand{\CC}{\mathbb C}
\newcommand{\HH}{\mathbb H}
\newcommand{\NN}{\mathbb N}
\newcommand{\PP}{\mathbb P}

\newcommand{\RR}{\mathbb R}
\newcommand{\ZZ}{\mathbb Z}
\newcommand{\cD}{\mathcal D}
\newcommand{\cA}{\mathcal A}
\newcommand{\cH}{\mathcal H}
\newcommand{\SL}{\mathop{\mathrm {SL}}\nolimits}
\newcommand{\orb}{\operatorname{orb}}
\newcommand{\w}{\operatorname{w}}

\newcommand{\Orth}{\operatorname{O}}
\newcommand{\Co}{\operatorname{Co}}

\newcommand{\bitem}{\begin{itemize}}
\newcommand{\eitem}{\end{itemize}}
\newcommand{\be}{\begin{equation}}
\newcommand{\ee}{\end{equation}}
\newcommand{\ba}{\begin{aligned}}
\newcommand{\ea}{\end{aligned}}

\begin{document}

\title[]{Conway invariant Jacobi forms on the Leech lattice}

\author{Kaiwen Sun}

\address{Korea Institute for Advanced Study, 85 Hoegiro, Dongdaemun-gu, Seoul, Korea}

\email{ksun@kias.re.kr}

\author{Haowu Wang}

\address{Center for Geometry and Physics, Institute for Basic Science (IBS), Pohang 37673, Korea}

\email{haowu.wangmath@gmail.com}

\subjclass[2020]{11F50}

\date{\today}

\keywords{Jacobi forms, Leech lattice, Conway group}

\begin{abstract}
In this paper we study Jacobi forms associated with the Leech lattice $\Lambda$ which are invariant under the Conway group $\Co_0$.  We determine and construct generators of modules of both weak and holomorphic Jacobi forms of integral weight and fixed index $t\leq 3$. As applications, (1) we find the modular linear differential equations satisfied by the holomorphic generators; (2) we determine the decompositions of many products of orbits of Leech vectors; (3) we calculate the intersections between orbits and Leech vectors; (4) we derive some conjugate relations among orbits modulo $t\Lambda$. 
\end{abstract}

\maketitle

\section{Introduction}
In 1985 Eichler and Zagier introduced the theory of Jacobi forms in their monograph \cite{EZ85}. These forms are holomorphic functions in two variables $(\tau, z) \in \HH \times \CC$ which are modular in $\tau$ with respective to $\SL_2(\ZZ)$ and quasi-periodic in $z$. Later, Gritsenko \cite{Gri88} defined Jacobi forms of lattice index by replacing $z$ with many variables associated with an integral positive-definite lattice. The Jacobi form creates an elegant bridge between different types of modular forms. For example, Jacobi forms can be identified as modular forms for the Weil representation of $\SL_2(\ZZ)$ through the theta decomposition, and are connected to modular forms on symmetric domains of type IV by the Fourier--Jacobi expansion (see \cite{Gri94}). Jacobi forms also have many applications in mathematical physics, such as the elliptic genera of some manifolds including $K3$ surfaces, and the topological string partition functions on various Calabi--Yau threefolds. It is a natural question to determine the structure of the space of Jacobi forms. This question was solved by Wirthm\"{u}ller \cite{Wir92} for Jacobi forms associated with root systems not of $E_8$ type, and by \cite{Wan21a, Wan21b, SW21} for the exceptional root system $E_8$. The ring of Jacobi forms indexed by lattices of rank two was recently described in \cite{WW21}. Little is known about spaces of Jacobi forms associated with other lattices, especially irreducible lattices of large rank. 

In this paper we investigate Jacobi forms on the Leech lattice which are invariant under the Conway group $\Co_0$, and give many applications to computational aspects of the Leech lattice. The Leech lattice is the unique even unimodular positive-definite lattice of rank $24$ which has no roots. It was discovered by Leech in 1967 \cite{Lee67}, and its uniqueness was proved by Conway in 1969 \cite{Con69a}. This lattice has many remarkable properties. For example, it plays a role in constructing the fake monster Lie algebra \cite{Bor90} and proving the monstrous moonshine conjecture \cite{Bor92}, and it achieves the densest sphere packing in $\RR^{24}$ \cite{CKMRV17}.  The group $\Co_0$ is the automorphism group of the Leech lattice, whose structure was first described by Conway \cite{Con69b}. The quotient of $\Co_0$ by its center gives a sporadic simple group of order $4,157,776,806,543,360,000$. Therefore, the Leech lattice is highly symmetrical, and we expect that the space of Conway invariant Jacobi forms will not be too large. Due to the importance of the Leech lattice and the Conway group, we also expect that Conway invariant Jacobi forms will have some applications in mathematics and physics.  These motivate us to study such Jacobi forms. 

We introduce the definition of Conway invariant Jacobi forms and state the main results. Let $\Lambda$ denote the Leech lattice equipped with bilinear form $(-,-)$.

\begin{definition}
Let $k\in \ZZ$ be an integer and $t\in \NN$ be a non-negative integer. If a holomorphic function $\varphi : \HH \times (\Lambda \otimes \CC) \rightarrow \CC$  satisfies the conditions
\begin{itemize}
\item[(i)] Conway invariance:
\begin{equation*}
\varphi(\tau, \sigma(\mathfrak{z}))=\varphi(\tau, \mathfrak{z}), \quad \sigma\in \Co_0,
\end{equation*}
\item[(ii)] Quasi-periodicity:
\begin{equation*}
\varphi (\tau, \mathfrak{z}+ x \tau + y)= \exp\left(-t\pi i (x,x)\tau - 2t\pi i (x,\mathfrak{z}) \right) \varphi ( \tau, \mathfrak{z} ), \quad x,y\in \Lambda,
\end{equation*}
\item[(iii)] Modularity:
\begin{equation*}
\varphi \left( \frac{a\tau +b}{c\tau + d},\frac{\mathfrak{z}}{c\tau + d} \right) = (c\tau + d)^k \exp\left( t\pi i \frac{c(\mathfrak{z},\mathfrak{z})}{c \tau + d}\right) \varphi ( \tau, \mathfrak{z} ), \quad \left( \begin{array}{cc}
a & b \\ 
c & d
\end{array} \right)   \in \SL_2(\ZZ),
\end{equation*}
\end{itemize}
and the Fourier expansion of $\varphi$ takes the form
\begin{equation*}
\varphi ( \tau, \mathfrak{z} )= \sum_{ n=0 }^\infty \sum_{ \ell \in \Lambda}f(n,\ell) q^n \zeta^\ell, \quad q=e^{2\pi i\tau}, \; \zeta^\ell = e^{2\pi i(\ell, \mathfrak{z})},
\end{equation*}
then it is called a \textit{Conway invariant weak Jacobi form} of weight $k$ and index $t$. If $\varphi$ further satisfies that $f(n,\ell)= 0$ whenever $2nt - (\ell,\ell) <0$, then it is called a \textit{Conway invariant holomorphic Jacobi form}. We denote the spaces of Conway invariant weak and holomorphic Jacobi forms of weight $k$ and index $t$ by $J_{k,\Lambda,t}^{\w,\Co_0}$ and $J_{k,\Lambda,t}^{\Co_0}$ respectively. 
\end{definition}

Jacobi forms of index $0$ are independent of $\mathfrak{z}$ and degenerate into usual modular forms on $\SL_2(\ZZ)$.
Since the Leech lattice $\Lambda$ is unimodular, the theta decomposition (see \cite[Corollary 2.6]{Gri94})) yields that every Conway invariant weak Jacobi form of weight $k$ and index $1$ is a holomorphic Jacobi form and can be expressed as $g(\tau)\Theta_\Lambda(\tau,\mathfrak{z})$,  where $g(\tau)$ is a modular form of weight $k-12$ on $\SL_2(\ZZ)$, and $\Theta_\Lambda(\tau,\mathfrak{z})$ is the  Jacobi theta function of the Leech lattice defined by
\begin{equation}\label{eq:theta}
    \Theta_{\Lambda}(\tau, \mathfrak{z}) = \sum_{\ell \in \Lambda} e^{\pi i (\ell,\ell)\tau + 2\pi i(\ell, \mathfrak{z})}. 
\end{equation}
In this paper we determine the generators of Conway invariant Jacobi forms of indices $2$ and $3$. Let $J_{*,\Lambda,t}^{\w, \Co_0}$ and $J_{*,\Lambda,t}^{\Co_0}$ denote the spaces of Conway invariant weak and holomorphic Jacobi forms of integral weight and given index $t$ respectively. Let $M_*(\SL_2(\ZZ))=\CC[E_4,E_6]$ be the ring of modular forms on $\SL_2(\ZZ)$. The following is our main theorem.

\begin{theorem}\label{MTH}
As free modules over $M_*(\SL_2(\ZZ))$, 
\begin{enumerate}
    \item $J_{*,\Lambda,2}^{\w, \Co_0}$ is generated by four forms of weights $-4$, $-2$, $0$, $0$. 
    \item $J_{*,\Lambda,2}^{\Co_0}$ is generated by four forms of weights $12$, $12$, $14$, $16$.
    \item $J_{*,\Lambda,3}^{\w, \Co_0}$ is generated by ten forms of weights $-14$, $-12$, $-12$, $-12$, $-10$, $-8$, $-6$, $-4$, $-2$, $0$.
    \item $J_{*,\Lambda,3}^{\Co_0}$ is generated by ten forms of weights $12$, $12$, $12$, $14$, $14$, $16$, $16$, $16$, $18$, $18$.  
\end{enumerate}
\end{theorem}

We sketch the main idea of the proof. We first use the differential operators approach in \cite{Wan21a} to estimate the minimal weight of weak Jacobi forms of a given index. Then we combine the arguments in \cite{Sak17, Wan21a, SW21} to construct generators. We also construct one of the singular-weight generators of  $J_{*,\Lambda,t}^{\Co_0}$ as the $t$-th Fourier--Jacobi coefficient of Borcherds' automorphic form $\Phi_{12}$ for the unimodular lattice $\mathrm{II}_{26,2}$ (see \cite{Bor95}). The main difficulty of the proof is to calculate the Fourier expansions of generators, because Conway invariant Jacobi forms have unwieldy Fourier expansions in $25$ variables. To overcome this difficulty, we write the Fourier expansion of a Jacobi form in terms of \textit{Conway orbits} defined as the $\Co_0$-invariant exponential polynomials
\begin{equation}\label{eq:orbit}
\orb(v) = \sum_{\sigma \in \Co_0  / (\Co_0)_v} e^{2\pi i(\sigma(v), \mathfrak{z})},   
\end{equation}
where $v \in \Lambda$ and $(\Co_0)_v$ is the stabilizer of $\Co_0$ with respect to $v$. The Conway orbits $\orb(v)$ of type $\frac{1}{2}(v,v)\leq 16$ are available in \cite{ATLAS}. In order to calculate the Fourier expansions of products of Jacobi forms, we have to know the decomposition of some products $\orb(v) \orb(u)$ into linear combinations of Conway orbits. We determine such non-trivial decompositions by comparing the Fourier--Jacobi expansion of $\Phi_{12}$ and the Borcherds denominator formula for the fake monster Lie algebra (see \cite{Bor90,Bor95}). Combining these arguments together, we prove the theorem. 

This paper is structured as follows. In \S \ref{sec:basic results} we collect and prove some basic results on Conway invariant Jacobi forms and orbits of Leech vectors. In \S \ref{sec:phi12} we introduce two methods to calculate the Fourier--Jacobi expansion of Borcherds' form $\Phi_{12}$, and determine the product decompositions of $\orb(v_2)\orb(v_2)$ and $\orb(v_2)\orb(v_3)$ by comparing the two methods, where $v_2$ and $v_3$ are Leech vectors of types $2$ and $3$. The main theorem is proved in \S \ref{sec:index2} for index $t=2$ and in \S \ref{sec:index3} for index $t=3$. In \S \ref{sec:applications} we present many applications of our main results. (i) We find the modular linear differential equations satisfied by some holomorphic Jacobi forms. (ii) We determine more product decompositions of Conway orbits by means of linear relations among Conway invariant holomorphic Jacobi forms of index $3$. These results are formulated in Appendix \ref{app:orbitprod}. (iii) We classify Conway invariant holomorphic Jacobi forms of singular weight $12$ and index $t\leq 3$ with non-trivial character. (iv) We use the Fourier expansions of our Jacobi forms to determine all conjugate relations among Conway orbits of type $\frac{1}{2}(v,v)\leq 16$ modulo $2\Lambda$ and $3\Lambda$. (v) We calculate the pullbacks of Conway invariant Jacobi forms and Conway orbits along Leech vectors of types $2$, $3$ and $4$. The pullbacks of Conway orbits are formulated in Appendix \ref{app:orbitinter}. In \S \ref{sec:high index} we discuss Conway invariant Jacobi forms of higher index and propose some open questions. We know from Borcherds' thesis \cite{Bor85} that $\Lambda / 4\Lambda$ have $31$ orbits with respect to $\Co_0$. We give an explicit description of the representative system of minimal length of $\Lambda / 4\Lambda$ in Theorem \ref{th:system4}. Borcherds' result yields that the rank of $J_{*,\Lambda,4}^{\w, \Co_0}$ is $31$. It seems very difficult to determine and construct the associated $31$ generators of index $4$.

\section{Conway invariant Jacobi forms}\label{sec:basic results}
In this section we prove some basic properties of Conway invariant Jacobi forms. Most of them are standard in the general theory of Jacobi forms. We also construct some basic Conway invariant holomorphic Jacobi forms. 

\subsection{Basic results} All results in this subsection are known to experts, and their analogues hold if we replace the Leech lattice and the Conway group with any even lattice and its orthogonal group. 
\begin{lemma}
There is no nonzero Conway invariant weak Jacobi form of odd weight. 
\end{lemma}
\begin{proof}
It follows from $-\mathrm{id} \in \Co_0$. 
\end{proof}

\begin{lemma}\label{lem:coefficients}
The Fourier expansion of $\varphi \in J_{k,\Lambda,t}^{\w, \Co_0}$ with $t\geq 1$ satisfies the following properties. 
\begin{enumerate}
    \item If $2n_1 t -(\ell_1,\ell_1) = 2n_2 t -(\ell_2,\ell_2)$ and $\ell_1 - \ell_2 \in t\Lambda$, then $f(n_1,\ell_1)=f(n_2,\ell_2)$.  In addition,
    $$
    f(n,\ell)=f(n,-\ell), \quad \text{for all $n\in\NN$ and $\ell\in \Lambda$.}
    $$
    \item If $f(n,\ell)\neq 0$, then 
    $$
    2n t - (\ell, \ell) \geq - \min\{ (v,v) : v \in \ell + t\Lambda \}.
    $$
\end{enumerate}
\end{lemma}
\begin{proof}
It is a particular case of \cite[Lemma 2.1]{Gri94}. The proof follows from the quasi-periodicity and the transformation under $-I_2 \in \SL_2(\ZZ)$.
\end{proof}
The above $(2)$ yields that for any $n\geq 0$ the $q^n$-term of a Conway invariant weak Jacobi form $\varphi$ defined by
$$
[\varphi]_{q^n} = \sum_{\ell \in \Lambda} f(n,\ell) \zeta^\ell, \quad \zeta^\ell = e^{2\pi i(\ell, \mathfrak{z})}
$$
has only finitely many terms, and is actually an exponential polynomial invariant under $\Co_0$.

\begin{lemma}\label{Lem:holomorphic}
Let $\varphi$ be a non-constant Conway invariant holomorphic Jacobi form of weight $k$ and index $t\geq 1$. Then $k\geq 12$. Moreover, if $\varphi$ is of weight $12$, then its Fourier expansion has the form
$$
\varphi(\tau, \mathfrak{z}) = \sum_{n=0}^\infty\ \sum_{\substack{ \ell \in \Lambda \\ (\ell,\ell)=2nt}} f(n,\ell)q^n \zeta^\ell. 
$$
\end{lemma}
\begin{proof}
The claim is proved by the theta decomposition (see e.g. \cite[Lemma 2.3, Lemma 2.5]{Gri94}).
\end{proof}
The minimal positive weight $12$ is called the \textit{singular} weight of Conway invariant holomorphic Jacobi forms. By the two lemmas above, the Fourier coefficients $f(n,\ell)$ of a Conway invariant holomorphic Jacobi form of singular weight and index $t$ only depend on the $\Co_0$-orbit of the class of $\ell$ in $\Lambda / t\Lambda$.

We have defined the Conway orbit of any Leech vector in \eqref{eq:orbit}. Obviously, any $q^n$-term of a Conway invariant Jacobi form is a $\CC$-linear combination of finitely many Conway orbits. In view of the $q^0$-term of a Conway invariant weak Jacobi form of index $t\geq 1$ (see Lemma \ref{lem:coefficients}), we consider the Leech vectors of minimal norm modulo $t\Lambda$. Let us define
\begin{equation}\label{eq:St}
  \mathcal{S}_t=\{ v \in \Lambda: (v,v) \leq (u,u), \; \text{for all} \; u \in v + t\Lambda \}.
\end{equation}
Due to the Conway invariance, we further consider the orbit space $\mathcal{S}_t/ \Co_0$.  If two orbits $\Co_0 v$ and $\Co_0 u$ in $\mathcal{S}_t/ \Co_0$ are conjugate modulo $t\Lambda$, namely there exists $\sigma\in\Co_0$ such that $v-\sigma(u)\in t\Lambda$, then $(v,v)=(u,u)$. The conjugate orbits of the same norm in $\mathcal{S}_t/ \Co_0$ correspond to the same Fourier coefficient in the sense of Lemma \ref{lem:coefficients}. Thus we define the quotient space $(\mathcal{S}_t/ \Co_0)/t\Lambda$, and denote the number of its elements by 
\begin{equation}\label{eq:rank-s}
s(t):=|(\mathcal{S}_t/ \Co_0)/t\Lambda|.    
\end{equation}
In other words, $s(t)$ is the number of $\Co_0$-orbits of the representative system of minimal length of $\Lambda / t\Lambda$. For convenience, we call the Conway orbits
\begin{equation}
\orb(v), \quad v \in \mathcal{S}_t / \Co_0    
\end{equation}
the \textit{basic Conway orbits of index $t$}, as they appear in the $q^0$-terms of weak Jacobi forms of index $t$.

\begin{lemma}\label{lem:weak-free}
The space $J_{*,\Lambda,t}^{\w, \Co_0}$ of Conway invariant weak Jacobi forms of integral weight and fixed index $t\geq 1$ is a free module of rank $s(t)$ over $M_*(\SL_2(\ZZ))$. Moreover, the $q^0$-terms of generators are $\CC$-linear combinations of basic Conway orbits of index $t$, and thus linearly independent over $\CC$. The space $J_{*,\Lambda,t}^{\Co_0}$ is also a free module of the same rank $s(t)$ over $M_*(\SL_2(\ZZ))$.
\end{lemma}
\begin{proof}
For any $\varphi \in J_{k,\Lambda,t}^{\w,\Co_0}$ the product $\Delta^{d}\varphi$ is a holomorphic Jacobi form for all integers $d$ satisfying
$$
d\geq  \frac{1}{2t}\cdot \max\{ (v,v): v\in \mathcal{S}_t \}.
$$
Thus the weight of Conway invariant weak Jacobi forms is bounded from below. Similar to the proof of \cite[Theorem 8.4]{EZ85}, we show that $J_{*,\Lambda,t}^{\w, \Co_0}$ and $J_{*,\Lambda,t}^{\Co_0}$ are free modules of the same rank over $M_*(\SL_2(\ZZ))$. Lemma \ref{lem:coefficients} yields that the $q^0$-term of a Conway invariant weak Jacobi form of index $t$ is a $\CC$-linear combination of the Conway orbits associated with vectors in $\mathcal{S}_t$. On the other hand, we view Conway invariant weak Jacobi forms of index $t$ as modular forms for the Weil representation attached to $\Lambda / t\Lambda$ (see \cite{Bor98}). By the obstruction principle in \cite{Bor99}, there exist weak Jacobi forms of sufficiently large weight whose $q^0$-term is a single basic Conway orbit of index $t$ (up to conjugate modulo $t\Lambda$). It follows that the rank is $s(t)$. The $q^0$-terms of generators are linearly independent, otherwise a suitable $\CC \cdot E_4^a E_6^b$-linear combination of generators will give a weak Jacobi form whose $q^0$-term is zero, and therefore can be written as the product of $\Delta$ and a $\CC[E_4,E_6]$-linear combination of generators. This contradicts the freeness of $J_{*,\Lambda,t}^{\w,\Co_0}$ as a $\CC[E_4,E_6]$-module.
\end{proof}

\begin{lemma}\label{lem:weak-holo}
For any even integer $k\geq 14$, the following identity holds
\begin{equation}
    \dim J_{k,\Lambda,t}^{\w, \Co_0} - \dim J_{k,\Lambda,t}^{\Co_0}=\delta_t, \quad t\geq 1,
\end{equation}
where $\delta_t$ is defined as
\begin{equation}
    \delta_t = \sum_{v\in (\mathcal{S}_t/\Co_0)/t\Lambda} \epsilon\left(\frac{(v,v)}{2t}\right),
\end{equation}
here $\epsilon(x):=\min\{ n\in \ZZ: n\geq x \}$. 
\end{lemma}
\begin{proof}
The proof is similar to that of \cite[Proposition 5.1]{SW21}. We only mention two essential ingredients of the proof. (a) A Conway invariant weak Jacobi form of index $t$ is a holomorphic Jacobi form if and only if its Fourier expansion does not contain the following terms
$$
q^n \orb(v), \quad 0\leq n < \frac{1}{2t}(v,v), \quad v \in (\mathcal{S}_t/\Co_0)/t\Lambda.
$$
(b) When the weight $k\geq 14$, the obstruction principle yields the existence of Conway invariant weak Jacobi forms of index $t$ whose Fourier expansion involves only one of the above terms. 
\end{proof}

\subsection{The construction of Jacobi forms} We introduce two standard methods to construct Jacobi forms. The first one is the differential operators which raise the weight of Jacobi forms. 

\begin{lemma}\label{lem:diffoperator}
Given a Conway invariant weak Jacobi form of weight $k$ and index $t\geq 1$
$$
\varphi(\tau,\mathfrak{z})=\sum_{n=0}^\infty \sum_{r \in \Lambda/\Co_0} f(n,r)q^n \cdot \orb(r).
$$
Then $H_{k}(\varphi)$ is a Conway invariant weak Jacobi form of weight $k+2$ and index $t$, where 
\begin{align*}
H_k(\varphi)(\tau,\mathfrak{z})&=\mathcal{H}(\varphi)(\tau,\mathfrak{z})+\frac{12-k}{12}E_2(\tau)\varphi(\tau,\mathfrak{z}),\\
\mathcal{H}(\varphi)(\tau,\mathfrak{z})&=\sum_{n\in \NN}\sum_{r \in \Lambda/\Co_0} \left(n-\frac{(r,r)}{2t} \right)f(n,r)q^n \cdot \orb(r),
\end{align*}
and $E_2(\tau)=1-24\sum_{n\geq 1}\sigma(n)q^n$ is the Eisenstein series of weight $2$ on $\SL_2(\ZZ)$.
\end{lemma}
\begin{proof}
The construction of $H_k$ relies heavily on the transformation laws of the heat operator $\mathcal{H}$ with respect to $\SL_2(\ZZ)$ which were described in \cite{EZ85, CK00}. We refer to \cite[Lemma 2.2]{Wan21IMRN} for a proof.
\end{proof}

Subsequently, if there is no confusion, we will write $H\varphi = H_k(\varphi)$, and denote the $d$-th composition of $H$ by $H^d$ for short. 

\begin{lemma}\label{lem:weight-0-identity}
The $q^0$-term of any $\varphi \in J_{0,\Lambda,t}^{\w, \Co_0}$ satisfies the following identity
\begin{equation*}
\sum_{r \in \Lambda/\Co_0} \left( 2t - (r,r)  \right) f(0,r) |\orb(r)|=0.
\end{equation*}
\end{lemma}
\begin{proof}
It is a particular case of \cite[Proposition 2.6]{Gri18}. It follows from that the constant term of $H_0(\varphi)(\tau,0)$ is zero.
\end{proof}

The second method is the Hecke operators which raise the index of Jacobi forms. 
\begin{lemma}\label{lem:index}
Let $m$ be a positive integer and $\varphi \in J^{\w, \Co_0}_{k,\Lambda,t}$. Then we have
$$ 
(\varphi\lvert_k T_{-}(m))(\tau,\mathfrak{z})=m^{-1}\sum_{\substack{ ad=m, a>0\\ 0\leq b<d }}a^{k}\varphi\left(\frac{a\tau+b}{d},a\mathfrak{z}\right) \in J^{\w ,\Co_0}_{k,\Lambda,mt},
$$
and the Fourier expansion of $\varphi\lvert_k T_{-}(m)$ is given by
$$
\left(\varphi\lvert_k T_{-}(m)\right)(\tau,\mathfrak{z})=\sum_{\substack{n\in \NN\\ \ell\in \Lambda}}\sum_{\substack{d\geq 1 \\ d| (n,\ell,m)}}d^{k-1}f\left(\frac{n m}{d^2},\frac{\ell}{d}\right)q^n\zeta^\ell\, ,
$$
where $f(n,\ell)$ are Fourier coefficients of $\varphi$, and $d | (n,\ell,m)$ means that $d|n$, $d|m$ and $d^{-1}\ell \in \Lambda$. 
\end{lemma}
\begin{proof}
It is a particular case of \cite[Corollary 1 of Proposition 4]{Gri88} or \cite[Corollary 2.9]{Gri94}.
\end{proof}

\subsection{Orbits of Leech vectors under the Conway group}\label{subsec:orbits}
Following the ATLAS \cite{ATLAS}, we define the \textit{type} of a Leech vector $v$ as half of its norm, i.e. $\frac{1}{2}(v,v)$. 
One can read the number $\lambda_n$ of Leech vectors of type $n$ from the theta function of the Leech lattice
\begin{equation}
\theta_\Lambda(\tau) = \sum_{\ell \in \Lambda} e^{\pi i (\ell,\ell)\tau} = \sum_{n=0}^\infty \lambda_n q^n. 
\end{equation}
We note that $\theta_\Lambda(\tau)=\Theta_\Lambda(\tau,0)$ (see \eqref{eq:theta}).
It is well-known that $\theta_\Lambda$ is a modular form of weight $12$ on $\SL_2(\ZZ)$, and it can be calculated via (see e.g. \cite{CS99})
\begin{equation}
\begin{aligned}
\theta_\Lambda(\tau) &= E_{12}(\tau) -  \frac{65520}{691}\Delta(\tau) = 1 + \frac{65520}{691} \sum_{n=1}^\infty (\sigma_{11}(n)-\tau(n))q^n.
\end{aligned}
\end{equation}
where $E_{12}(\tau)=1+O(q)$ is the normalized Eisenstein series of weight $12$, and 
\begin{equation}\label{eq:Delta}
\Delta(\tau) = q\prod_{n=1}^\infty(1-q^n)^{24} = \sum_{n=1}^\infty \tau(n)q^n  \end{equation}
is the unique normalized cusp form of weight $12$ on $\SL_2(\ZZ)$.

The Conway group $\Co_0$ acts naturally on Leech vectors. The orbits of $\Lambda/\Co_0$ were described in ATLAS \cite[Page 181]{ATLAS} for Leech vectors of type $x\leq 16$. There are in total $44$ orbits of type $x\leq 16$ including the orbit of type $0$. Among the $44$ orbits, the orbits of the same type have distinct numbers of elements, and their numbers of elements are all divisible by $p=65520$. Let $O_{x,y p}$ stand for an orbit of type $x$ whose number of elements is $y p$. If there are multiple orbits of the same type, then we mark them as $O_{xa,y_1p}$, $O_{xb,y_2p}$, etc. in the order of the ATLAS.   We formulate the $43$ nonzero orbits as follows 
\begin{align*}
&O_{2,3p}& &O_{3,256p}& &O_{4,6075p}& &O_{5,70656p}& &O_{6a,518400p}&\\ &O_{6b,6900p}& &O_{7,2861568p}& &O_{8a,3p}&
&O_{8b,12295800p}& &O_{8c,141312p}& \\
&O_{9a,12441600p}& &O_{9b,32972800p}& &O_{10a,143078400p}& &O_{10b,279450p}&
&O_{10c,1430784p}&\\
&O_{11a,19430400p}& &O_{11b,393465600p}&
&O_{12a,256p}& &O_{12b,141312p}&  &O_{12c,12441600p}&\\ 
&O_{12d,2049300p}& &O_{12e,393465600p}& &O_{12f,667699200p}& &O_{13a,12441600 p}& &O_{13b,1007271936 p}&\\
&O_{13c,1573862400 p}& &O_{14a,286156800 p}& &O_ {14b,5508518400 p}&  &O_{14c,13800 p}& &O_{14d,19430400 p}&\\ 
&O_{14e,49183200 p}& &O_{15a,2861568 p}& &O_{15b,1335398400 p}& &O_{15c,19430400 p}& &O_{15d,8012390400 p}&\\
&O_{15e,3147724800 p}& &O_{16a,6075 p}& &O_{16b,12441600 p}& &O_{16c,286156800 p}& &O_{16d,98366400 p}&\\
&O_{16e,5901984000 p}& &O_{16f,16024780800 p}& &O_{16g,3147724800 p}&   \end{align*}
Since the numbers $y$ appearing in the above orbits are very large, for the sake of simplicity, we omit the subscript $yp$ of the orbit. We also use the same notation $O_{x}$ to denote the corresponding Conway orbit
$\sum_{v\in O_{x}} e^{2\pi i(v, \mathfrak{z})}$. 
There are $3$ non-primitive orbits of type $x \leq 16$, that is, $O_{8a}=\orb(2v_2)$, $O_{12a}=\orb(2v_3)$ and $O_{16a}=\orb(2v_4)$, where $v_x$ is a Leech vector of type $x$. For future use, we denote by $O_{18a}=\orb(3v_2)$ the non-primitive Conway orbit of type $18$.

To describe the representative system of $\Lambda / t\Lambda$, we introduce the following notions. 
\begin{definition}\label{def:weights}
For $v,u\in \Lambda$, we say that $v$ is conjugate to $u$ modulo $t\Lambda$ if $v-u\in t\Lambda$, noted by $v \sim_t u$. The $t$-weight of a Leech vector $v$ is defined as the number
$$
\#\{ u\in \Lambda : u \in t\Lambda + v, \quad (u,u)=(v,v) \}.
$$
The $t$-weight of an orbit $\Co_0 v$ is defined as the number 
$$
\#\{ u\in \Co_0 v: u \in t\Lambda + v\}.
$$
\end{definition}
We warn that the $t$-weight of an orbit may be different from the $t$-weight of a vector in this orbit (see Theorem \ref{th:system4}). 
The following lemma was proved in \cite{Con69a}.
\begin{lemma}
Representatives of $\Lambda/2\Lambda$ of minimal norm may be found among vectors of types up to $4$, according to the weighted equality (here the $2$-weights of orbits and vectors coincide)
$$
1+\frac{|O_2|}{2} + \frac{|O_3|}{2} + \frac{|O_4|}{48} = 2^{24}.
$$
\end{lemma}

The following lemma was proved in \cite[Theorem 4.1]{Mar02}. It can also be derived from \cite[Figure 2]{Bor85}.
\begin{lemma}
Representatives of $\Lambda/3\Lambda$ of minimal norm may be found among vectors of types up to $9$, according to the weighted equality (here the $3$-weights of orbits and vectors coincide)
$$
1+ |O_2| +  |O_3| + |O_4| + |O_5| + |O_{6a}| +  \frac{|O_{6b}|}{3} + \frac{|O_7|}{2} + \frac{|O_{8b}|}{9} + \frac{|O_{9b}|}{36} = 3^{24}.
$$
\end{lemma}
We know from \cite[Lemma 4.3]{Mar02} that $O_{8a}$, $O_{8c}$ and $O_{9a}$ are respectively conjugate to $O_2$, $O_5$ and $O_{6a}$ modulo $3\Lambda$.  It is easy to derive the following result from the lemmas above. 
\begin{lemma}\label{lem:data}
\noindent
\begin{enumerate}
    \item The rank of $J_{*,\Lambda,2}^{\w,\Co_0}$ over $M_*(\SL_2(\ZZ))$ is $4$. The basic Conway orbits of index $2$ are $O_0$, $O_2$, $O_3$, $O_4$. The number $\delta_2$ defined in Lemma \ref{lem:weak-holo} is $5$.
    \item The rank of $J_{*,\Lambda,3}^{\w,\Co_0}$ over $M_*(\SL_2(\ZZ))$ is $10$. The basic Conway orbits of index $3$ are $O_0$, $O_2$, $O_3$, $O_4$, $O_5$, $O_{6a}$, $O_{6b}$, $O_7$, $O_{8b}$, $O_{9b}$. The number $\delta_3$ defined in Lemma \ref{lem:weak-holo} is $19$.
\end{enumerate}
\end{lemma}

\subsection{Conway invariant holomorphic Jacobi forms} We construct some Conway invariant holomorphic Jacobi forms of low weights. As in \cite{Sak17} we start with the Jacobi theta function of $\Lambda$. As we mentioned in the introduction, the Jacobi theta function $\Theta_\Lambda(\tau,\mathfrak{z})$ defined in \eqref{eq:theta} is a Conway invariant holomorphic Jacobi form of weight $12$ and index $1$. By applying the Hecke operators introduced in Lemma \ref{lem:index} to $\Theta_\Lambda(\tau,\mathfrak{z})$, we can construct Conway invariant holomorphic Jacobi forms of weight $12$ and arbitrary index. More precisely, for any $t\geq 2$ we have 
\begin{equation}
\begin{aligned}
A_t(\tau, \mathfrak{z}):=&\,\frac{1}{\sigma_{11}(t)}(\Theta_\Lambda | T_{-}(t))(\tau,\mathfrak{z}) \\
=&\, 1 + \frac{1}{\sigma_{11}(t)}\sum_{\substack{\ell \in \Lambda, (\ell,\ell)=2t}} q\cdot \zeta^\ell + O(q^2) \in J_{12,\Lambda,t}^{\Co_0}.
\end{aligned}
\end{equation}
It is easy to see from Lemma \ref{lem:index} that every $\sigma_{11}(t)A_t$ has integral Fourier coefficients. 
For convenience, we will use $A_1$ instead of $\Theta_\Lambda$ from now on. The reduction of $A_t$ is given by
\begin{equation}
A_t(\tau,0)= E_{12}(\tau) - \frac{65520}{691}\frac{\tau(t)}{\sigma_{11}(t)} \Delta(\tau).     
\end{equation}
Using the data of Conway orbits of type $x\leq 16$, we can calculate the $q^d$-terms of the Fourier expansion of $A_t$ for $d\leq 16/t$. The forms $A_2$ and $A_3$ will be used later and we present their Fourier expansions up to $q^5$-terms
\begin{equation}
\begin{aligned}
A_2=\,&1+\frac{1}{2049}\Big(O_{2} q+O_{4} q^2+(O_{6 a}+O_{6 b}) q^3+(2049 O_{8 a}+O_{8 b}+O_{8 c}) q^4\\
&+(O_{10 a}+O_{10 b}+O_{10 c}) q^5\Big)+O(q^6),\\
A_3=\,&1+\frac{1}{177148}\Big(O_{3} q+(O_{6 a}+O_{6 b}) q^2+(O_{9 a}+O_{9 b}) q^3+(O_{12 a}+O_{12 b}+O_{12 c}+O_{12 d}\\
&+O_{12 e}+O_{12 f}) q^4+(O_{15 a}+O_{15 b}+O_{15 c}+O_{15 d}+O_{15 e}) q^5\Big)+O(q^6).
\end{aligned}    
\end{equation}

Since $A_t$ has singular weight, the differential operators introduced in Lemma \ref{lem:diffoperator} acting on them give zero. In order to construct Conway invariant holomorphic Jacobi forms of weight $14$ and index $t\geq 2$, we use the trick in \cite{Sak17, Wan21a}. We notice that $A_1(t\tau,t\mathfrak{z})$ is a Conway invariant holomorphic Jacobi form of weight $12$ and index $t$ with respect to the congruence subgroup $\Gamma_0(t)$ of $\SL_2(\ZZ)$. Given a modular form $f_{k}(\tau)$ of weight $k$ on $\Gamma_0(t)$, then $f_k(\tau)A_1(t\tau,t\mathfrak{z})$ defines a Conway invariant holomorphic Jacobi form of weight $12+k$ and index $t$ for $\Gamma_0(t)$. The trace operation
\begin{equation}
\mathrm{Tr}_{\SL_2(\ZZ)}(f_k(\tau)A_1(t\tau,t\mathfrak{z}))=\sum_{\gamma\in \Gamma_0(t)\backslash \SL_2(\ZZ)}(f_k(\tau)A_1(t\tau,t\mathfrak{z}))|_{12+k,t}\gamma
\end{equation}
gives a Conway invariant holomorphic Jacobi form of weight $12+k$ and index $t$, where $\lvert_{k,t}\gamma$ is the slash action of $\gamma\in \SL_2(\ZZ)$ defined via
$$ 
(\phi \lvert_{k,t}\gamma)(\tau,\mathfrak{z}):=(c\tau + d)^{-k} \exp\left(- t\pi i \frac{c(\mathfrak{z},\mathfrak{z})}{c \tau + d}\right) \phi \left( \frac{a\tau +b}{c\tau + d},\frac{\mathfrak{z}}{c\tau + d} \right). 
$$
It is well-known that $f_2(\tau):=(tE_2(t\tau)-E_2(\tau))/(t-1) \in M_2(\Gamma_0(t))$. We apply the trace operator to $f_2(\tau)A_1(t\tau,t\mathfrak{z})$ and denote the normalized image by
\begin{equation}
B_t(\tau, \mathfrak{z}) =  1 + O(q) \in J_{14,\Lambda, t}^{\Co_0}, \quad t\geq 2,   
\end{equation}
whose reduction is
\begin{equation}
B_t(\tau,0) = E_{14}(\tau)=E_4(\tau)^2E_6(\tau).     
\end{equation}
More precisely, for prime $t$ we have
\begin{equation}\label{Bt}
B_t(\tau, \mathfrak{z})=\frac{t^{12}}{t^{12}-1}\Big( f_2(\tau)A_1(t\tau,t\mathfrak{z})-\frac{1}{t^{13}}\sum_{j=0}^{t-1}f_2\Big(\frac{\tau+j}{t} \Big)A_1\Big( \frac{\tau+j}{t},\mathfrak{z} \Big) \Big).
\end{equation}
The Fourier expansions of $B_2$ and $B_3$ are as follows
\begin{equation}
\begin{aligned}
B_2=\,&1+\frac{1}{4095}\Big((98280-O_{2}) q+(98280-24 O_{2}-24 O_{3}-O_{4}) q^2+(393120-24 O_{2}\\
&-96 O_{3}-24 O_{4}-24 O_{5}-O_{6 a}-O_{6 b}) q^3+(98280-96 O_{2}-144 O_{3}-24 O_{4}\\
&-96 O_{5}-24 O_{6 a}-24 O_{6 b}-24 O_{7}+4095 O_{8 a}-O_{8 b}-O_{8 c}) q^4+(589680-24 O_{2}\\
&-192 O_{3}-96 O_{4}-144 O_{5}-24 O_{6 a}-24 O_{6 b}-96 O_{7}+98280 O_{8 a}-24 O_{8 b}\\[-1mm]
&-24 O_{8 c}-24 O_{9 a}-24 O_{9 b}-O_{10 a}-O_{10 b}-O_{10 c}) q^5\Big) +O(q^6).
\end{aligned}    
\end{equation}

\begin{equation}
\begin{aligned}
B_3=\,& 1+\frac{1}{531440}\Big( (6377280-12 O_{2}-O_{3}) q+ (19131840-84 O_{2}-12 O_{3}-36 O_{4}-12 O_{5}\\
&-O_{6 a}-O_{6 b})q^2+ (6377280-96 O_{2}-36 O_{3}-72 O_{4}-84 O_{5}-12 O_{6 a}-12 O_{6 b}\\
&-36 O_{7}-12 O_{8 a}-12 O_{8 b}-12 O_{8 c}-O_{9 a}-O_{9 b})q^3+ (44640960-216 O_{2}-12 O_{3}\\
&-180 O_{4}-96 O_{5}-36 O_{6 a}-36 O_{6 b}-72 O_{7}-84 O_{8 a}-84 O_{8 b}-84 O_{8 c}-12 O_{9 a}\\
&-12 O_{9 b}-36 O_{10 a}-36 O_{10 b}-36 O_{10 c}-12 O_{11 a}-12 O_{11 b}-O_{12 a}-O_{12 b}-O_{12 c}\\
&-O_{12 d}-O_{12 e}-O_{12 f})q^4+ (38263680-168 O_{2}-84 O_{3}-144 O_{4}-216 O_{5}-12 O_{6 a}\\
&-12 O_{6 b}-180 O_{7}-96 O_{8 a}-96 O_{8 b}-96 O_{8 c}-36 O_{9 a}-36 O_{9 b}-72 O_{10 a}-72 O_{10 b}\\
&-72 O_{10 c}-84 O_{11 a}-84 O_{11 b}-12 O_{12 a}-12 O_{12 b}-12 O_{12 c}-12 O_{12 d}-12 O_{12 e}\\
&-12 O_{12 f}-36 O_{13 a}-36 O_{13 b}-36 O_{13 c}-12 O_{14 a}-12 O_{14 b}-12 O_{14 c}-12 O_{14 d}\\[-1mm]
&-12 O_{14 e}-O_{15 a}-O_{15 b}-O_{15 c}-O_{15 d}-O_{15 e})q^5\Big)+O(q^6).
\end{aligned}    
\end{equation}

\section{Jacobi forms of singular weight and the fake monster Lie algebra}\label{sec:phi12}
In 1990 Borcherds \cite{Bor90} constructed a celebrated generalized Kac--Moody algebra whose root lattice is $\mathrm{II}_{25,1}$, that is, the unique even unimodular lattice of signature $(25,1)$. This infinite-dimensional Lie algebra describes the physical states of a bosonic string moving on the $\mathbb{Z}_2$ orbifold of the torus $\RR^{25,1}/\mathrm{II}_{25,1}$ (see the string theory background in \cite{dixon1988beauty}), and is called the \textit{fake monster Lie algebra}.  Borcherds proved that the denominator identity of this algebra defines an automorphic form of weight $12$ on a symmetric domain of type IV and dimension $26$. This form is exceptional, because Scheithauer \cite{Sch17} proved that it is the unique holomorphic Borcherds product of singular weight on unimodular lattices. In this section we construct Conway invariant holomorphic Jacobi forms of singular weight and calculate their Fourier expansions using this remarkable identity. 

We first review the modularity of this denominator identity. 
Let $U$ be an even unimodular lattice of signature $(1,1)$. Then $U\oplus \Lambda$ gives a model of $\mathrm{II}_{25,1}$. We further define $\mathrm{II}_{26,2}=2U\oplus \Lambda$. The symmetric domain $\cD(\mathrm{II}_{26,2})$ of type IV attached to $\mathrm{II}_{26,2}$ is one of the two conjugate connected components of the space
$$
\{ [\mathcal{Z}] \in \PP(\mathrm{II}_{26,2}\otimes\CC): (\mathcal{Z},\mathcal{Z})=0, (\mathcal{Z},\overline{\mathcal{Z}})<0 \}.
$$
Let $\Orth^+(\mathrm{II}_{26,2})$ denote the orthogonal group which preserves $\mathrm{II}_{26,2}$ and $\cD(\mathrm{II}_{26,2})$. We define the affine cone over $\cD(\mathrm{II}_{26,2})$ as
$$
\cA(\mathrm{II}_{26,2}) = \{ \mathcal{Z} \in \mathrm{II}_{26,2}\otimes\CC : [\mathcal{Z}] \in \cD(\mathrm{II}_{26,2}) \}.
$$
In 1995 Borcherds constructed a holomorphic function $\Phi_{12}: \cA(\mathrm{II}_{26,2}) \to \CC$ which satisfies
\begin{align*}
\Phi_{12}(t\mathcal{Z}) &= t^{-12}\Phi_{12}(\mathcal{Z}), \quad \text{for all $t\in \CC^\times$},\\
\Phi_{12}(g(\mathcal{Z})) &= \det(g) \Phi_{12}(\mathcal{Z}), \quad \text{for all $g\in \Orth^+(\mathrm{II}_{26,2})$,}
\end{align*}
and vanishes precisely with multiplicity one on hyperplanes orthogonal to roots of $\mathrm{II}_{26,2}$. 
In other words, $\Phi_{12}$ is a modular form of minimal weight $12$ and determinant character for $\Orth^+(\mathrm{II}_{26,2})$. This function has an infinite product expansion. The modular variety $\cD(\mathrm{II}_{26,2})/\Orth^+(\mathrm{II}_{26,2})$ has a unique $0$-dimensional cusp and twenty-four $1$-dimensional cusps which correspond to the twenty-four Niemeier lattices. 
We fix a basis $U=\ZZ e+\ZZ f$ with $e^2=f^2=0$ and $(e,f)=-1$.
At the $1$-dimensional cusp corresponding to the Leech lattice we realize $\cD(\mathrm{II}_{26,2})$ as the tube domain
$$
\cH(\Lambda) = \{ Z=-\omega e+ \mathfrak{z} - \tau f : \tau, \omega \in \HH, \; \mathfrak{z}\in \Lambda\otimes\CC: (\mathrm{Im}(Z),\mathrm{Im}(Z))<0 \}.
$$
We write $\alpha=ne+\ell+mf \in U\oplus \Lambda$ as $(n,\ell,m)$. Note that $\alpha^2=-2nm+(\ell,\ell)$ and $(\alpha,Z)=n\tau+m\omega + (\ell,\mathfrak{z})$.
Borcherds \cite{Bor90, Bor95} proved that $\Phi_{12}$ are represented on $\cH(\Lambda)$ by
\begin{equation}\label{eq:denominator}
\begin{aligned}
\Phi_{12}(Z)&=e^{2\pi i (\rho, Z)}\prod_{\alpha\in\Delta_{+}}\left( 1 - e^{2\pi i (\alpha,Z)} \right)^{p_{24}(1-(\alpha,\alpha)/2)}\\
&=\sum_{\sigma\in W} \det(\sigma) \sum_{n=1}^\infty \tau(n) e^{2\pi i (\sigma(n\rho),Z)},
\end{aligned}    
\end{equation}
where $\rho=(1,0,0)$ is the Weyl vector, $W$ is the reflection group of $U\oplus \Lambda$, $\Delta_{+}$ is the set of positive roots of the fake monster Lie algebra which are either positive multiples of $\rho$ or vectors $\alpha$ in $U\oplus \Lambda$ satisfying $(\alpha,\alpha)\leq 2$ and $(\alpha,\rho)<0$, $\tau(n)$ is the Ramanujan tau function in \eqref{eq:Delta}, and $p_{24}(n)$ is defined by 
$$
\Delta(\tau)^{-1}=\sum_{n=0}^\infty p_{24}(n)q^{n-1}. 
$$
This is the so-called \textit{Weyl--Kac--Borcherds denominator identity} of the fake monster Lie algebra. In terms of Conway invariant Jacobi forms, one can express $\Phi_{12}$ as follows (see \cite{Bor95, Gri18})
\begin{equation}\label{eq:FJ12}
\begin{aligned}
\Phi_{12}(Z) &= \Delta(\tau)\cdot \exp\Big(-\sum_{m=1}^\infty ((\Delta^{-1}A_1)| T_{-}(m))(\tau,\mathfrak{z})e^{2\pi i m\omega}  \Big) \\
&= \sum_{m=0}^\infty \Phi_{12,m}(\tau,\mathfrak{z}) e^{2\pi i m\omega},
\end{aligned}
\end{equation}
where $T_{-}(m)$ are Hecke operators introduced in Lemma \ref{lem:index}. 

\begin{theorem}\label{th:phi}
For any $m\in \NN$, the function $\Phi_{12,m}$ is a Conway invariant holomorphic Jacobi form of weight $12$ and index $m$. Moreover, its Fourier expansion takes the form
$$
\Phi_{12,m}(\tau,\mathfrak{z}) = \sum_{n=0}^\infty \sum_{\substack{\ell \in \Lambda\\ (\ell,\ell)=2nm}} c_m(n,\ell) q^n \zeta^\ell
$$
and satisfies the following properties.
\begin{enumerate}
    \item For any $n,m\in \NN$ and $\ell \in \Lambda$ we have
    $$
    c_m(n,\ell)= - c_n(m,\ell).
    $$
    In particular, $c_m(m,\ell)=0$ for any $m\in\NN$ and $\ell \in \Lambda$.
    \item For any $n,m\in\NN$ and $\ell\in \Lambda$, we have $c_m(n,\ell)\in \ZZ$. Let $d$ be the largest positive integer such that $n/d, m/d \in \NN$ and $d^{-1}\ell \in \Lambda$. Then we have
    $$
    c_m(n,\ell) \in \{ \tau(d), -\tau(d), 0 \}.
    $$
\end{enumerate}
\end{theorem}
\begin{proof}
The function $\Phi_{12,m}$ is in fact the $m$-th Fourier--Jacobi coefficient of $\Phi_{12}$ at the Leech cusp, which follows that it is a Conway invariant holomorphic Jacobi form of weight $12$ and index $m$. Since $\Phi_{12,m}$ has singular weight, Lemma \ref{Lem:holomorphic} yields the claimed Fourier expansion. We notice that $\Phi_{12}$ has the determinant character. Thus $\Phi_{12}$ is anti-invariant under exchanging $\tau$ and $\omega$, which implies the property $(1)$. The property $(2)$ follows from the denominator identity \eqref{eq:denominator}. 
\end{proof} 

It is clear that $\Phi_{12,0}=\Delta$ and $\Phi_{12,1}=-A_1$.
We derive from the above theorem that $c_m(0,0)=-c_0(m,0)=-\tau(m)$ and $c_m(1,\ell)=-c_1(m,\ell)=1$. Thus when $m\geq 2$ the Fourier expansion of $\Phi_{12,m}$ starts with
\begin{equation}
\Phi_{12,m}(\tau,\mathfrak{z})=-\tau(m) + \sum_{\ell \in \Lambda, (\ell,\ell)=2m} q \cdot\zeta^\ell + O(q^2),   
\end{equation}
and the reduction of $\Phi_{12,m}$ is given by
\begin{equation}\label{eq:reduction of phi}
\Phi_{12,m}(\tau,0) = -\tau(m)E_{12}(\tau) + \frac{65520}{691}\sigma_{11}(m)\Delta(\tau). 
\end{equation}

By comparing reductions of $A_t$ and $\Phi_{12,t}$, we prove the following lemma.

\begin{lemma}
When $t\geq 2$, the forms $A_t$ and $\Phi_{12,t}$ are linearly independent. In particular, we have
$$
\dim J_{12,\Lambda,t}^{\Co_0} \geq 2, \quad \text{for $t\geq 2$}.
$$
\end{lemma}

By Theorem \ref{th:phi} and the reduction to modular forms of weight $12$ on $\SL_2(\ZZ)$, we can unambiguously determine the Fourier expansion of $\Phi_{12,t}$ to very high $q$ orders. For example, we find
\begin{equation}\label{c2order8}
\begin{aligned}
\Phi_{12,2}=\,&24+O_{2} q+O_{6 b} q^3+(24 O_{8 a}+O_{8 c}) q^4+(O_{10 b}+O_{10 c}) q^5+(24 O_{12 a}+O_{12 b}+O_{12 c}) q^6\\
&+(O_{14 c}+O_{14 d}+O_{14 e}) q^7+(24 O_{16 a}+O_{16 b}+O_{16 c}) q^8+O(q^9),\\
\Phi_{12,3}=\,&-252 + O_{3}q - O_{6 b}q^2  + 
  (O_{12 a} + O_{12 b} -O_{12 d})q^4+  (O_{15 a} -O_{15 c})q^5+O(q^6),\\
\Phi_{12,4}=\,&1472+O_{4} q-q^2 (24 O_{8 a}+O_{8 c})-q^3 (O_{12 a}+O_{12 b}-O_{12 d})+O(q^5).
\end{aligned}    
\end{equation}
Let us briefly show how to find the above $\Phi_{12,2}$. We first determine its reduction by \eqref{eq:reduction of phi} as
\begin{equation}
\begin{aligned}
\Phi_{12,2}(\tau,0)=&\,24+196560 q+452088000 q^3+9263479680 q^4+112054531680 q^5\\
&+824834949120 q^6+4496467248000 q^7+19573719984000 q^8+O(q^{9}).
\end{aligned}    
\end{equation}
Since $\Phi_{12,2}$ is a holomorphic Jacobi form of index $2$ and singular weight, its $q^n$-term is a linear combination of Conway orbits of type $2n$. Clearly, the $q^1$-term has to be $O_2$. To determine the $q^3$-term, we write $[\Phi_{12,2}]_{q^3}=x O_{6a}+y O_{6b}$. Then 
$$
[\Phi_{12,2}]_{q^3}(\tau,0)=x|O_{6a}|+y|O_{6b}|, \quad \text{i.e.} \quad  33965568000x+452088000y=452088000.
$$
By Theorem \ref{th:phi}, $x,y$ can only be $\pm1$ or 0. Obviously, the only solution is $x=0$ and $y=1$. Following this procedure, we uniquely determine the Fourier expansion of $\Phi_{12,2}$ up to $q^8$-term. For $[\Phi_{12,2}]_{q^8}$, the similar linear equation has two solutions: the one in \eqref{c2order8} and $24(O_{16a}+O_{16b})$, but the second possibility is ruled out knowing only orbits of type $\mathrm{orb}(2v)$ have coefficient $24$ which can only be $O_{16a}$. We also use the same trick to find the Fourier expansions of $\Phi_{12,3}$ up to $q^5$-term and $\Phi_{12,4}$ up to $q^4$-term.

It turns out that these $\Phi_{12,t}$ are very useful
to determine the product decomposition of Conway orbits, which would be very difficult to compute in a brutal way due to the huge size of the orbits. For example, from \eqref{eq:FJ12} it is straightforward to compute
\begin{equation}\label{c2Jacobi}
\begin{aligned}
\Phi_{12,2}=&\, \frac{1}{2}\frac{A_1^2}{\Delta}-\Delta \Big[\Big(\frac{A_1}{\Delta}\Big)\Big|T_-(2)\Big]\\
=&\,24 +  O_{2}q + 
  \Big(\frac12 O_{2}\otimes O_2-98280 - 2300 O_{2}  - 276 O_{3} - 23 O_{4} -
    O_{5} - \frac12 O_{8 a}\Big)q^3\\
&+ (12 O_{2}\otimes O_2+ O_{2}\otimes O_{3}-2358720 - 102304 O_{2}  - 17802 O_{3}  - 2600 O_{4}\\
&- 299 O_{5}  - 24 O_{6 a} - 24 O_{6 b}- O_{7} + 12 O_{8 a})q^4+O(q^5).
\end{aligned}    
\end{equation}
By comparing the $q^3$-, $q^4$-terms of \eqref{c2order8} and \eqref{c2Jacobi}, we can neatly determine the decompositions of orbit products $O_{2}\otimes O_2$ and $O_{2}\otimes O_{3}$. We then achieve the following result.
\begin{lemma}\label{lem:O2O2-O2O3}
The decompositions of $O_{2}\otimes O_2$ and $O_{2}\otimes O_{3}$ in Appendix \ref{app:orbitprod} hold. 
\end{lemma}

By comparing the $q^5$-term of $\Phi_{12,2}$ in \eqref{c2order8} and \eqref{c2Jacobi}, we further determine the decomposition of $O_{3}\otimes O_3+2O_{2}\otimes O_4$, but cannot separate this combination. Using the $q^6$-term of $\Phi_{12,2}$, we determine $O_{3}\otimes O_4+O_{2}\otimes O_5$, but cannot separate them either. Utilizing these constraints, we can compute from \eqref{eq:FJ12} that
\begin{equation}\label{c3Jacobi}
\begin{aligned}
\Phi_{12,3}=&-\frac{1}{6}\frac{A_1^3}{\Delta^2}+{A_1}\Big[ \Big(\frac{A_1}{\Delta}\Big)\Big|T_-(2)\Big]-\Delta \Big[\Big(\frac{A_1}{\Delta}\Big)\Big|T_-(3)\Big]\\
=&-252 + O_{3}q - O_{6 b}q^2  +\frac13\big(2O_{2}\otimes  O_5-O_{2}\otimes O_{6b}+O_{2}\otimes O_{8a} - 89609 O_{2}- 151800 O_{3}\\
& - 93196 O_{4} - 
 38350 O_{5}- 11868 O_{6 a}- 9484 O_{6 b} - 2804 O_{7}  + 
 4600 O_{8 a} - 484 O_{8 b}\\
 &- 276 O_{8 c} - 23 O_{9 a}  - 
 66 O_{9 b} - 3 O_{10 a}+ 43 O_{10 b} + 
 3 O_{12 a}    -O_{18a}
\big)q^4+O(q^5).
\end{aligned}    
\end{equation}
By comparing the $q^4$-term of $\Phi_{12,3}$ in \eqref{c2order8} and \eqref{c3Jacobi}, we determine the decomposition of the combination $O_{2}\otimes( 2 O_5-O_{6b}+O_{8a})$. These partially determined decompositions of orbit products can be remedied by the technique of Conway invariant Jacobi forms of index $3$, which will be discussed later. The new technique even allows us to determine more product decompositions of orbits, which does not appear in the combinations here, such as $O_2\otimes O_{6a}$.

\section{Conway invariant Jacobi forms of index 2}\label{sec:index2}
In this section we prove parts $(1)$ and $(2)$ of Theorem \ref{MTH}.
We first determine the space of Conway invariant holomorphic Jacobi forms of singular weight and index $2$. 
\begin{lemma}\label{lem:singular2}
$$
\dim J_{12,\Lambda,2}^{\Co_0} =2.
$$
\end{lemma}
\begin{proof}
Suppose $\dim J_{12,\Lambda,2}^{\Co_0} >2$. We notice that there is only one Conway orbit of type $2$. By virtue of $A_2$ and $\Phi_{12,2}$, we can cancel the $q^0$- and $q^1$-terms of the third holomorphic Jacobi form of weight $12$ and index $2$. Then Lemma \ref{Lem:holomorphic} yields that there exists a Jacobi form
$$
\phi(\tau,\mathfrak{z})=q^n\sum_{\ell\in\Lambda, (\ell,\ell)=4n} f(n,\ell)\zeta^\ell + O(q^{n+1})
$$
for some positive integer $n\geq 2$. Since $\phi/\Delta^2$ defines a weak Jacobi form of weight $-12$ and index $2$, we have $\phi(\tau,0)=0$. By Lemma \ref{lem:coefficients}, the $q^n$-term of $\phi$ is a linear combination of basic Conway orbits of type $2n$ and index $2$. Therefore, Lemma \ref{lem:data} implies that $n=2$ and up to nonzero scalar $\phi(\tau,\mathfrak{z})=q^2 O_4+O(q^3)$, which contradicts with $\phi(\tau,0)=0$. 
\end{proof}

\begin{theorem}
The free $\CC[E_4,E_6]$-module $J_{*,\Lambda,2}^{\Co_0}$ is generated by $A_2$, $\Phi_{12,2}$, $B_2$ and $HB_2$. 
\end{theorem}
\begin{proof}
We know that the minimal weight of non-constant Jacobi forms in $J_{*,\Lambda,2}^{\Co_0}$ is $12$.  According to Lemma \ref{lem:data}, $J_{*,\Lambda,2}^{\Co_0}$ has four generators. By the above lemma, there are exactly two generators of weight $12$. Clearly, $B_2$ has to be a generator of weight $14$. The image of $B_2$ under the differential operator introduced in Lemma \ref{lem:diffoperator} gives a Conway invariant holomorphic Jacobi form of weight $16$ and index $2$. We verify that $HB_2$ is independent of $E_4A_2$ and $E_4\Phi_{12,2}$. Therefore, $HB_2$ is a generator of weight $16$. We have found all four generators, and thus proved the theorem. 
\end{proof}

We remark that $A_1^2$ can be expressed by $A_2, \Phi_{12,2}, B_2, HB_2$ as
\begin{equation}\label{idlevel2}
A_1^2=2\Delta \Phi_{12,2}+\frac{455}{1024}E_4(E_6B_2-3E_4HB_2)+\frac{683}{18432}(5 E_4^3+4 E_6^2)A_2.
\end{equation}

To determine the free module of weak Jacobi forms, we first estimate the minimal weight of weak Jacobi forms using the differential operators approach established in \cite{Wan21a}.

\begin{lemma}\label{lem:minimal2}
$$
J_{k,\Lambda,2}^{\w,\Co_0}=\{0\} \quad \text{if $k<-4$}.
$$
\end{lemma}
\begin{proof}
Suppose that there exists a nonzero $\phi_k\in J_{k,\Lambda,2}^{\w,\Co_0}$ for some $k<-4$ whose $q^0$-term is nonzero. We can assume that $k=-8$ or $-6$, otherwise we multiply $\phi_k$ with a modular form in $\CC[E_4,E_6]$. We write the $q^0$-term of $\phi_k$ as
$$
[\phi_k]_{q^0} = c_0 O_0 + c_2 O_2 + c_3 O_3 + c_4 O_4, \quad c_0, c_2, c_3, c_4 \in \CC.
$$
Applying the differential operators to $\phi_k$ we construct weak Jacobi forms $H\phi_k$, $H^2\phi_k$, ... of weights $k+2$, $k+4$, .... For each weak Jacobi form of negative weight including $\phi_{k}$, its reduction to $\mathfrak{z}=0$ is identically zero, and in particular its $q^0$-term reduces to zero by taking $\mathfrak{z}=0$, which yields a linear equation with four unknowns $c_i$. For example, for $\phi_k$ we have $c_0+c_2|O_2|+c_3|O_3|+c_4|O_4|=0$. For the weak Jacobi form of weight $0$, there is a similar linear equation obtained from Lemma \ref{lem:weight-0-identity}. In this way, we build a system of $|k|/2 + 1$ linear equations with $4$ unknowns $c_i$. A direct calculation shows that this linear system has only trivial solution for $k=-6$ and $-8$. This contradicts our assumption that the $q^0$-term of $\phi_k$ is nonzero. 
\end{proof}

In the following we construct four weak Jacobi forms $\varphi_{k,2}$ of weight $k$ and index $2$ in terms of holomorphic generators and present their $q^0$- and $q^1$-terms.
\begin{equation}
    \begin{aligned}
\varphi_{-4,2}=&\,\frac{1}{32\Delta^2}\Big(E_4^2(2813277 A_2 -2048 \Phi_{12,2}) +1842750 (E_6B_2 +15 E_4 HB_2)\Big)\phantom{iiiiiiiiiiii--}\\
=& \,  (8491392000 - 2025 O_{3} + 64 O_{4})+ 
  (543449088000 - 2073600 O_{2} \\
 &- 380700 O_{3} + 33792 O_{4} - 
    2025 O_{5} + 64 O_{6 a})q
+O(q^2).
    \end{aligned}
\end{equation}
\begin{equation}
    \begin{aligned}
 \varphi_{-2,2}=&\,6H\varphi_{-4,2}\\
 =&\,\frac{1}{8\Delta^2}\Big(E_4 (921375 E_4B_2 -2813277 E_6A_2 +2048 E_6\Phi_{12,2}  )-11056500 E_6 HB_2\Big)\phantom{iiiiiiii} \\
 =& \, (67931136000 + 2025 O_{3} - 
   256 O_{4}) + (5977939968000 - 16588800 O_{2}\\
 & - 1514700 O_{3} + 
    55296 O_{4} + 2025 O_{5} - 256 O_{6 a}) q
+O(q^2). 
    \end{aligned}
\end{equation}
\begin{equation}
    \begin{aligned}
 \varphi_{0,2}=&\,3H\varphi_{-2,2}\\
 =&\,\frac{1}{16\Delta^2}\Big((3 E_4^3+2 E_6^2)(2813277 A_2 -2048\Phi_{12,2} )+\!1842750 E_4 (21 E_4 HB_2-4 E_6B_2) \Big)\\
=& \,  (237758976000 - 2025 O_{3} + 
   640 O_{4}) + (33150394368000 - 58060800 O_{2} \\
 &- 3199500 O_{3} + 
    49152 O_{4} - 2025 O_{5} + 640 O_{6 a}) q
+O(q^2). 
    \end{aligned}
\end{equation}
\begin{equation}
    \begin{aligned}
 \psi_{0,2}=\frac{683}{\Delta}(\Phi_{12,2}-24A_2)
 = 675 O_2 + 8(2025 O_2 -  O_4) q 
+O(q^2). \phantom{iiiiiiiiiiiiiiiiiiiiiiiiiiiii}
    \end{aligned}
\end{equation}

\begin{theorem}\label{th:weak3}
The free $\CC[E_4,E_6]$-module $J_{*,\Lambda,2}^{\w, \Co_0}$ is generated by $\varphi_{-4,2}$, $\varphi_{-2,2}$, $\varphi_{0,2}$ and $\psi_{0,2}$. 
\end{theorem}
\begin{proof}
As we mentioned in the proof of Lemma \ref{lem:weak-free}, for any $\phi_k\in J_{k,\Lambda,2}^{\w, \Co_0}$, the product $\Delta^2 \phi_k$ is a holomorphic Jacobi form of weight $k+24$, and therefore lies in the $\CC[E_4,E_6]$-module generated by $A_2$, $\Phi_{12,2}$, $B_2$ and $HB_2$. In this way, we find that the vector space $J_{k,\Lambda,2}^{\w,\Co_0}$ for $k=-4$, $-2$ and $0$ has dimension $1$, $1$ and $3$ respectively. Thus there are one generator of weight $-4$, one generator of weight $-2$ and two generators of weight $0$. We then finish the proof. 
\end{proof}

We remark that the argument in the proof of Lemma \ref{lem:minimal2} also yields the estimations
$$
\dim J_{-4,\Lambda,2}^{\w,\Co_0} \leq 1, \quad \dim J_{-2,\Lambda,2}^{\w,\Co_0} \leq 2, \quad \dim J_{0,\Lambda,2}^{\w,\Co_0} \leq 3. 
$$
This gives another proof of the above theorem. As far as the authors know, $J_{*,\Lambda,2}^{\w,\Co_0}$ is the first free module of weak Jacobi forms on an irreducible lattice which has multiple generators of weight $0$.

In summary, the generating series of Conway invariant weak Jacobi forms of index $2$ is given by
\begin{equation}
\begin{aligned}
\sum_{k\in\ZZ} \dim J_{k,\Lambda, 2}^{\w,\Co_0}x^k =&\, \frac{x^{-4}+x^{-2}+2}{(1-x^4)(1-x^6)}\\
=&\, x^{-4}+x^{-2}+3
+2 x^2+4 x^4+4 x^6+5 x^8+5 x^{10}+7 x^{12}+6 x^{14}+8 x^{16}\\
&+8 x^{18}+9 x^{20}
+9 x^{22}+11 x^{24}+10 x^{26}+12 x^{28}+12 x^{30}
+O(x^{32}).
\end{aligned}    
\end{equation}

\section{Conway invariant Jacobi forms of index 3}\label{sec:index3}
In this section we prove parts $(3)$ and $(4)$ of Theorem \ref{MTH}. 
\subsection{Conway invariant holomorphic Jacobi forms of index 3}
We first estimate the minimal weight of Conway invariant weak Jacobi forms of index $3$ in a similar way to the case of index $2$. 
\begin{lemma}\label{lem:minimal3}
$$
J_{k,\Lambda,3}^{\w,\Co_0}=\{0\} \quad \text{if $k<-14$}.
$$
\end{lemma}
\begin{proof}
Suppose that there exists a nonzero $\phi_k\in J_{k,\Lambda,3}^{\w,\Co_0}$ for some $k<-14$ whose $q^0$-term is nonzero. As explained in the proof of Lemma \ref{lem:minimal2}, we only need to consider the cases of $k=-18$ and $-16$. We represent the $q^0$-term of $\phi_k$ as
$$
[\phi_k]_{q^0} = c_0 O_0 + c_2 O_2 + c_3 O_3 + c_4 O_4 + c_5 O_5 + c_{6a}O_{6a} + c_{6b}O_{6b} + c_7 O_7 + c_8 O_{8b} + c_9 O_{9b}, \quad c_i \in \CC.
$$
Again, by applying the differential operators to $\phi_k$ we construct Jacobi forms $H\phi_k$, $H^2\phi_k$, ... of weights $k+2$, $k+4$, .... By taking $\mathfrak{z}=0$ and considering the reductions of weak Jacobi forms of non-positive weights,  we construct a system of $|k|/2 + 1$ linear equations with $10$ unknowns $c_i$. By direct calculation,  we find that the only non-trivial solution of this linear system is
$$
c_{0}, c_2, c_3, c_4, c_5, c_7, c_8, c_9 = 0, \quad c_{6a}|O_{6a}|+c_{6b}|O_{6b}|=0.
$$
Thus the $q^0$-term of $\phi_k$ has the form
$$
\phi_k = |O_{6b}|O_{6a} - |O_{6a}|O_{6b}+O(q).
$$
We notice that $\Delta^2\phi_k$ defines a Conway invariant holomorphic Jacobi form of weight $k+24$ and index $3$. It follows that $k+24\geq 12$, i.e. $k\geq -12$. This contradicts our assumption on the weight $k$. We have thus proved the lemma.
\end{proof}

For calculation purposes, we need to determine the decompositions of $O_2\otimes O_4$ and $O_3\otimes O_3$. The following identity is useful. 

\begin{lemma}
The following identity holds.
\begin{equation}\label{O2O4}
\begin{aligned}
H(A_1B_2)=&A_1HB_2-\frac{E_4}{1326780} \Big(44287  (7 E_4^3+9 E_6^2)A_3+797160 E_4E_6B_3\\
&-1296 A_1 \Phi_{12,2} -3888 \Delta \Phi_{12,3}\Big) +\frac{292}{243}\left(2 E_4^3+E_6^2\right) HB_3 \\
&-\frac{584}{81} E_4 E_6 H^2B_3+\frac{2}{4095} \big(2049 H^2(A_1A_2)-4 H^2(A_1\Phi_{12,2})\big).
\end{aligned}    
\end{equation}
\end{lemma}
\begin{proof}
By means of the decomposition of $O_2\otimes O_2$ which was determined in Lemma \ref{lem:O2O2-O2O3}, we are able to calculate the Fourier expansions of the basic forms $A_1A_2$, $A_1\Phi_{12,2}$, $A_1B_2$ and $A_1HB_2$ up to (and including) $q^3$-terms. These data are enough to verify that both sides of \eqref{O2O4} have the same Fourier coefficients up to $q^3$-terms. Therefore, the quotient of their difference by $\Delta^4$ gives a Conway invariant weak Jacobi form of weight $-20$ and index $3$, which has to be zero by Lemma \ref{lem:minimal3}. We then prove the desired identity. 
\end{proof}

\begin{lemma}\label{lem:O2O4-O3O3}
The decompositions of $O_2\otimes O_4$ and $O_3\otimes O_3$ in Appendix \ref{app:orbitprod} hold.
\end{lemma}
\begin{proof}
As mentioned in \S \ref{sec:phi12}, we can determine the product decomposition
\begin{equation}\label{OO33twoOO24}
\begin{aligned}
O_{3}\otimes O_3+2O_{2}\otimes O_4=&\,16773120 + 1140156 O_{2} + 354800 O_{3} + 97152 O_{4} + 
 22356 O_{5}\\
 &+ 4094 O_{6 a} + 
 4598 O_{6 b} + 552 O_{7}+ 48 O_{8 b} + 2 O_{9 a}    + 2 O_{9 b}\\
 &+ 2 O_{10 b} + 2 O_{10 c}+ O_{12 a}.    
\end{aligned}    
\end{equation}
The $q^4$-terms of $A_1A_2$, $A_1B_2$ and $A_1HB_2$ contain $O_{2}\otimes O_4$, but not $O_{3}\otimes O_3$. To determine the decomposition of $O_{2}\otimes O_4$, one first needs to make an ansatz. From the triangle inequality, we see
\begin{equation}\label{pqbound}
O_p\otimes O_q=\sum c_{n*}O_{n*},\qquad \frac{1}{2}(\sqrt{2p}-\sqrt{2q})^2\le n \le \frac{1}{2}(\sqrt{2p}+\sqrt{2q})^2.   
\end{equation}
Then $O_2\otimes O_4 $ at most contains orbits of type from 2 to 11, while $
O_3\otimes O_3 $ at most contains orbits of type from 0 to 12.
In fact, we can do better knowing only the orbits in $O_{3}\otimes O_3+2O_{2}\otimes O_4$ can appear in the ansatz. In summary, we have the following ansatz:
\begin{equation}\label{OO2433}
O_2\otimes O_4=\sum x_* O_{*} ,\qquad
O_3\otimes O_3= 16773120 +O_{12a}+\sum y_* O_{*}.    
\end{equation}
The above sums take over orbits $O_{2}, O_{3} , O_{4} ,O_{5}
,O_{6 a} , O_{6 b} , O_{7}, O_{8 b}, O_{9 a}  , O_{9 b}, O_{10 b}, O_{10 c}$.
Substitute the above ansatz for $O_2\otimes O_4$ into $A_1A_2$, $A_1B_2$, $A_1HB_2$, then the $q^4$-term of \eqref{O2O4} solves
\begin{equation}
x_2= 93150, x_3 = 48600, x_4 = 16192, x_5= 4050, x_{6a}= 759,
x_{6b}= 891,  x_7= 100, x_{8b}= 8.    
\end{equation}
It remains to determine $x_{9a},x_{9b},x_{10b},x_{10c}$. Using $|O_2\otimes O_4|=|O_2|\times |O_4|$, we can solve
\begin{equation}
x_{10c}=\frac{26175}{2944}-\frac{ 200 }{23}x_{9a}-\frac{5600 }{243}x_{9b}-\frac{25 }{128}x_{10b}.    
\end{equation}
As all $x_*$ and $y_*$ need to be non-negative integers, \eqref{OO33twoOO24} and \eqref{OO2433} require that $x_{9a},x_{9b},x_{10b},x_{10c}$ can only be 0 or 1. It is then easy to find that the only solution is
\begin{equation}
x_{9a}=1,\; x_{9b}=0,\; x_{10b}=1,\; x_{10c}=0.    
\end{equation}
This solves the decomposition of $O_2\otimes O_4$, and by extension $O_3\otimes O_3$.
\end{proof}

Thanks to the decompositions of $O_2\otimes O_2$, $O_2\otimes O_3$ and $O_2\otimes O_4$, we are able to calculate the Fourier expansions of $A_1A_2$, $A_1\Phi_{12,2}$, $A_1B_2$ and $A_1HB_2$ up to (and including) $q^4$-terms. This allows
us to construct a new Conway invariant holomorphic Jacobi form of weight $12$ and index $3$, and prove that this new form  generates $J_{12,\Lambda,3}^{\Co_0}$ together with $A_3$ and $\Phi_{12,3}$. 

\begin{lemma}
The following function lies in the space $J_{12,\Lambda,3}^{\Co_0}$.
\begin{equation}\label{psi123}
\begin{aligned}
\Psi_{12,3}=&\,84909\Phi_{12,3}+\frac{1}{3888\Delta}\Big(42576 (486 \psi_{24,3} + 33215 ( E_4 E_6B_3 - 6 E_4^2 HB_3 + 12 E_6 H^2B_3) )\\
  &+ 10345968 A_1 (17 \Phi_{12,2}-2049 A_2) + 
  44287 A_3 (18153 E_4^3 + 10231 E_6^2)
  \Big)\\
=&\,(48384O_{6b}-644O_{6a})q^2+(-644O_{9a}+243 O_{9b})q^3+ 
 q^4 (-644 O_{12 c} + 48384 O_{12 d}\\
 & - 644 O_{12 e} + 
    243 O_{12 f})+ 
 q^5 (-644 O_{15 b} + 48384 O_{15 c} + 243 O_{15 d} - 
    644 O_{15 e})+O(q^6),
\end{aligned}    
\end{equation}
where
\begin{equation}
\psi_{24,3}=\frac{1}{E_4} \Big(8190 A_1 HB_2 + 2049 H^2(A_1A_2) - 16 H^2(A_1\Phi_{12,2})   \Big).
\end{equation}
\end{lemma}
\begin{proof}
We first calculate the Fourier coefficients of $\Psi_{12,3}$ up to $q^3$-terms by definition. If $\Psi_{12,3}$ has no poles, then we conclude from its Fourier expansion that $\Psi_{12,3}$ is a holomorphic Jacobi form. We now prove that $\Psi_{12,3}$ has no poles. This is equivalent to show that $\psi_{24,3}$ is holomorphic, which follows from the following identity
\begin{equation*}
\begin{aligned}
E_6\psi_{24,3}=&-\frac{A_1}{24}  (59421 E_6A_2+151515 E_4B_2+424 E_6\Phi_{12,2})-\frac{44287}{314928}A_3 (85041 E_4^3 E_6+19583 E_6^3)\\
&-\frac{33215}{13122} B_3 \left(2199 E_4^4+4340 E_4 E_6^2\right)-\frac{66430}{2187} H^2B_3 \left(3249 E_4^3+3290 E_6^2\right)-53 \Delta E_6\Phi_{12,3}\\
&-55323 H^3(A_1A_2)+\frac{217192885}{2187} E_4^2 E_6 HB_3+\frac{34833}{2} E_4
H(A_1A_2).    
\end{aligned}
\end{equation*}
We explain why the above identity holds. Let $\psi_{30,3}$ denote the right hand side of this identity. We check that $E_6E_4\psi_{24,3}-E_4\psi_{30,3}=O(q^5)$. Thus the quotient by $\Delta^5$ defines a Conway invariant weak Jacobi form of weight $-26$ and index $3$, which has to be zero by Lemma \ref{lem:minimal3}. This proves our claim, and thus proves the lemma. To calculate the $q^4$- and $q^5$-terms of $\Psi_{12,3}$, we solve the linear equations (defined by the reductions of $q^4$- and $q^5$-terms to $\mathfrak{z}=0$)
$$
\sum c_* |O_{12*}| = 0 \quad \text{and} \quad \sum c_* |O_{15*}| = 0.
$$
These coefficients have to be $0$, $-644$, $48384$ or $243$, as the orbits of types $12$ and $15$ are conjugate to $O_3$, $O_{6a}$, $O_{6b}$ or $O_{9b}$ modulo $3\Lambda$. 
\end{proof}

\begin{lemma}\label{lem:singular3}
The space $J_{12,\Lambda,3}^{\Co_0}$ is generated by $A_3$, $\Phi_{12,3}$ and $\Psi_{12,3}$. 
\end{lemma}
\begin{proof}
Suppose $\dim J_{12,\Lambda,3}^{\Co_0} > 3$. Similar to Lemma \ref{lem:singular2}, by means of $A_3$, $\Phi_{12,3}$ and $\Psi_{12,3}$, we can cancel the $q^0$-, $q^1$- and $q^2$-terms of the fourth Jacobi form  of singular weight and index $3$ denoted $\phi$. Then its Fourier expansion takes the form $q^3O_{9b}+O(q^4)$, which contradicts $\phi(\tau,0)=0$. 
\end{proof}

We further construct a form $\Psi_{14,3}\in J_{14,\Lambda,3}^{\Co_0}$ which is linearly independent of $B_3$, and a form $\Psi_{16,3}\in J_{16,\Lambda,3}^{\Co_0}$ which is linearly independent of $E_4A_3,E_4\Phi_{12,3},E_4\Psi_{12,3}$, $HB_3$ and $H\Psi_{14,3}$.
\begin{equation}
\begin{aligned}
\Psi_{14,3}=&\,\frac{1}{1109777760\Delta}\Big(14 E_4 (E_4 (99645 B_3 E_4 + 88574 A_3 E_6)+597870 (2 E_4 H^2B_3\!-\!E_6 HB_3))\\
&-1990170 A_1B_2+729 (4098 H(A_1A_2)-H(A_1\Phi_{12,2}))\Big)\\
=&\,1+\frac{1}{6850480}\Big(2 (41102880-3102 O_{2}+29 O_{3})q+(246617280-43428 O_{2}\phantom{iiiiiiiiiiiiiii}\\
&+696 O_{3}-948 O_{4}-129 O_{5}-5 O_{6b}-11 O_{6a})q^2+(82205760-49632 O_{2}\\
&+2088 O_{3}-1896 O_{4}-903 O_{5}-60 O_{6b}-132 O_{6a}-462 O_{7}-6204 O_{8a}\\
&-129 O_{8c}-156 O_{8b}-11 O_{9a}-14 O_{9b})q^3\Big)+O(q^4).
\end{aligned}    
\end{equation}
\begin{equation*}
\begin{aligned}
\phantom{iiiiii}\Psi_{16,3}=&\,2660529 E_4\Phi_{12,3}+\frac{1}{3888\Delta}\Big(-69984 \big(1825659 H^2(A_1A_2)+123364 H^2(A_1\Phi_{12,2})\big)\\
&\,+1594320 \big(278111  E_4^2 E_6B_3-6  (174161 E_4^3+103950 E_6^2)HB_3+3337332 E_4 E_6 H^2B_3\big)\\
&\,-3888  {A_1} (107980251  E_4{A_2}-1239877 \Phi_{12,2} E_4-1360490040 {HB_2})\\
&\,+44287 E_4A_3 (3734541 E_4^3+5165011  E_6^2)\Big)\\
=&\,616(3974400 O_{2} - 186300 O_{3} + 11776 O_{4} - 675 O_{5} + 
 1728 O_{6 b}
)q^2\\
&\,+(9792921600 O_{2}- 1032847200 O_{3} + 101556224 O_{4} - 
 7068600 O_{5} - 3060288 O_{6 a} \\
&\,+ 242694144 O_{6 b}  - 30800 O_{7}+ 21504 O_{8 b}- 415800 O_{8 c} - 2673 O_{9 b} 
)q^3+O(q^4).
\end{aligned}    
\end{equation*}

\begin{lemma}\label{lem:weight-14}
The space $J_{14,\Lambda,3}^{\Co_0}$ is generated by $B_3$ and $\Psi_{14,3}$. \end{lemma}
\begin{proof}
We know that $E_4A_3, E_4\Phi_{12,3}$, $E_4\Psi_{12,3}$, $\Psi_{16,3}$, $HB_3$ and $H\Psi_{14,3}$ are linearly independent. Suppose $\dim J_{14,\Lambda,3}^{\Co_0}\geq 3$. Then $\dim J_{20,\Lambda,3}^{\Co_0}\geq 9$. 
Therefore, there exists a nonzero $\phi\in J_{20,\Lambda,3}^{\Co_0}$ whose Fourier expansion takes the form 
$$
\phi= q^2(c_0+c_2 O_2+c_3 O_3)+O(q^3).
$$
This gives a weak Jacobi form
$$
\psi:= \Delta^{-2}\phi= c_0+c_2 O_2+c_3 O_3 + O(q) \in J_{-4,\Lambda,3}^{\w,\Co_0}.
$$
We apply the differential operators to construct $H\psi$ and $H^2\psi$. The reductions to $\mathfrak{z}=0$ of these two Jacobi forms together with $\psi$ yields a system of $3$ linear equations with $3$ unknowns $c_0$, $c_2$ and $c_3$. We find that this system has only trivial solution. Thus $\psi/\Delta$ gives a nonzero Conway invariant weak Jacobi form of weight $-16$ and index $3$, which contradicts Lemma \ref{lem:minimal3}. 
\end{proof}

\begin{theorem}
The free $\CC[E_4,E_6]$-module $J_{*,\Lambda,3}^{\Co_0}$ is generated by ten forms 
$$
A_3,\; \Phi_{12,3},\; \Psi_{12,3},\; B_3,\; \Psi_{14,3},\; HB_3,\; H\Psi_{14,3},\; \Psi_{16,3},\; H^2B_3,\; H^2\Psi_{14,3}
$$
which have weights $12$, $12$, $12$, $14$, $14$, $16$, $16$, $16$, $18$, $18$ respectively. 
\end{theorem}
\begin{proof}
We know from Lemma \ref{lem:data} that this free module has ten generators. By Lemma \ref{lem:singular3} and Lemma \ref{lem:weight-14}, there are exactly three generators of weight $12$ and two generators of weight $14$. Since $\Psi_{16,3}$, $HB_3$ and $H\Psi_{14,3}$ are linearly independent with $E_4A_3, E_4\Phi_{12,3}$ and $E_4\Psi_{12,3}$, they have to be generators of weight $16$. We also verify that $H^2B_3$ and $H^2\Psi_{14,3}$ are linearly independent of $E_6A_3, E_6\Phi_{12,3}$, $E_6\Psi_{12,3}$, $E_4B_3$ and $E_4\Psi_{14,3}$, so they are generators of weight $18$. We have found all ten generators. This completes the proof of the theorem. 
\end{proof}

\subsection{Conway invariant weak Jacobi forms of index 3}
We now construct Conway invariant weak Jacobi forms of index $3$ in terms of the ten holomorphic generators. For the lowest weight weak Jacobi form $\varphi_{-14,3}$, we normalize the constant term in its $q^0$-term to be $1$.
\begin{equation}
\begin{aligned}
\varphi_{-14,3}=&\,\frac{1}{8692012800\Delta^3}\Big(\!-\!\frac{E_4E_6}{4881764160} (6737646323408 A_3+333264204 \Phi_{12,3}+62607 \Psi_{12,3})\!\!\\
&-\frac{E_4^2}{36} (324193 B_3+25407 \Psi_{14,3})+ \frac{E_6}{5765760} (3 \Psi_{16,3}-26961014080 HB_3\\
&+32854902080 H\Psi_{14,3})+{E_4}  (35843 H^2B_3+166557 H^2\Psi_{14,3}) \Big),\\
=&\,\Big(1-\frac{O_{2}}{8190}+\frac{O_{3}}{199680}-\frac{O_{4}}{3159000}+\frac{O_{5}}{55111680}-\frac{O_{7}}{5208053760}\\
&\,+\frac{O_{8b}}{38362896000}-\frac{O_{9b}}{540094464000}\Big)+O(q).
\end{aligned}    
\end{equation}
We remark that this $q^0$-term is consistent with the unique solution of the linear system built in a similar way to Lemma \ref{lem:minimal3} for $k=-14$. For weight $-12$, two obvious weak Jacobi forms are 
\begin{equation}
\begin{aligned}
H\varphi_{-14,3}=&\,\Big(\frac{13}{6}-\frac{O_{2}}{5460}+\frac{7 O_{3}}{1198080}-\frac{O_{4}}{3790800}+\frac{O_{5}}{110223360}+\frac{O_{7}}{31248322560}\\
&-\frac{O_{8 b}}{76725792000}+\frac{O_{9 b}}{648113356800}\Big)+O(q),
\end{aligned}    
\end{equation}
and
\begin{equation}
\begin{aligned}
\psi_{-12,3} = \frac{\Psi_{12,3}}{\Delta^2}=(48384O_{6b}-644O_{6a})+O(q).
\end{aligned}    
\end{equation}
We construct the third independent weak Jacobi form of weight $-12$ such that its $q^0$-term does not contain $O_0$ and $O_{6a}$:
\begin{equation}
\begin{aligned}
\phi_{-12,3}= &\,\frac{1}{579467520\Delta^3}\Big(\frac{E_4^3}{3973760}(16703964459728 A_3+277004364 \Phi_{12,3}+178351 \Psi_{12,3})\\
& - \frac{E_6^2}{35763840}(  202943145966736A_3+  25403868\Phi_{12,3}+  2307555\Psi_{12,3})\\
&-91E_4E_6 (248711 B_3 + 27289 \Psi_{14,3})+ \frac{E_4^2}{3080} (79754795120 HB_3\\
&+54331156880 H\Psi_{14,3} + 3 \Psi_{16,3} )+364E_6 (1682869 H^2B_3-8469 H^2\Psi_{14,3})\Big) \\
=&\,\Big(O_{2}-\frac{19 O_{3}}{256}+\frac{17 O_{4}}{2025}-\frac{25 O_{5}}{23552}+\frac{13 O_{6 b}}{1380}\!-\frac{11 O_{7}}{953856}+\frac{O_{8 b}}{1366200}\!-\frac{O_{9 b}}{32972800}\Big)\!+\!O(q).
\end{aligned}    
\end{equation}

\begin{theorem}\label{3weak}
The free $\CC[E_4,E_6]$-module $J_{*,\Lambda,3}^{\w, \Co_0}$ is generated by ten forms 
$$
\varphi_{-14,3}, \quad \phi_{-12,3}, \quad \psi_{-12,3}, \quad H^d\varphi_{-14,3}, \quad 1\leq d\leq 7
$$
whose weights are respectively $-14$, $-12$, $-12$, $-14+2d$, for $1\leq d\leq 7$. 
\end{theorem}
\begin{proof}
For any $\phi_k\in J_{k,\Lambda,3}^{\w, \Co_0}$ we have $\Delta^3\phi_k\in J_{36+k,\Lambda,3}^{\Co_0}$. We note that the Fourier expansions of ten holomorphic generators can be calculated up to (and including) $q^3$-terms. These data are sufficient to determine the subspace
$$
\mathcal{J}_k:=\{ f\in J_{36+k,\Lambda,3}^{\Co_0}: f=O(q^3)\}
$$
which is isomorphic to $J_{k,\Lambda,3}^{\w, \Co_0}$ via the map $f\mapsto f/\Delta^3$. Therefore, for any even $-14\leq k \leq 0$ we can find a basis of $J_{k,\Lambda,3}^{\w, \Co_0}$ with known $q^0$-terms. We then pick up the ten weak generators from these bases and determine their weights. Notice the $q^0$-term of $\varphi_{-14,3}$ does not contain orbits of type $6$, then the $q^0$-terms of all $H^d\varphi_{-14,3}$ do not contain orbits of type $6$. On the other hand, the $q^0$-terms of both $\phi_{-12,3}$ and $\psi_{-12,3}$ do contain orbits of type $6$ and are linearly independent. Then it is easy to verify that the $q^0$-terms of the ten forms are linearly independent over $\CC$. Therefore, the ten forms are indeed generators. This concludes the proof.
\end{proof}

In summary, the generating series of Conway invariant weak Jacobi forms of index $3$ is given by
\begin{equation}
\begin{aligned}
\sum_{k\in\ZZ} \dim J_{k,\Lambda, 3}^{\w,\Co_0}x^k =&\, \frac{x^{-14}+3x^{-12}+{x^{-10}}+x^{-8}+x^{-6}+x^{-4}+x^{-2}+1}{(1-x^4)(1-x^6)}\\
=&\, x^{-14}+3x^{-12}+2x^{-10}+5x^{-8}+6x^{-6}+7x^{-4}+9x^{-2} +12+11 x^2\\
&+15 x^4+16 x^6+17 x^8+19 x^{10}+22 x^{12}+21 x^{14}+25 x^{16}+26 x^{18}\\
&+27 x^{20} +29 x^{22}+32 x^{24}+31 x^{26}+35 x^{28}+36 x^{30}+O(x^{32}).
\end{aligned}    
\end{equation}

By Theorem \ref{3weak}, the two-dimensional vector space of Conway invariant weak Jacobi forms of weight $-10$ and index $3$ are spanned by $E_4\varphi_{-14,3}$ and $H^2\varphi_{-14,3}$. Thus $H\phi_{-12,3}$ should be linearly dependent of them. Indeed, we find
\begin{equation}
H\phi_{-12,3}=4680H^2\varphi_{-14,3}  -20280E_4\varphi_{-14,3}.    
\end{equation}

\section{Applications}\label{sec:applications}
\subsection{Modular linear differential equations}
Modular forms on $\SL_2(\ZZ)$ satisfy some modular linear differential equations (MLDEs) composed of Serre derivatives. As a generalization, Jacobi forms can satisfy MLDEs composed of differential operators introduced in Lemma \ref{lem:diffoperator}. The MLDEs for holomorphic Jacobi forms are particularly interesting.
For example, all holomorphic Jacobi forms $\phi$ of singular weight (including $A_t,\Phi_{12,t}$ and $\Psi_{12,3}$) satisfy the first-order MLDE
\begin{equation}
H\phi=0.    
\end{equation}

We further study MLDEs for holomorphic Jacobi forms of non-singular weight. 
We first conclude that $B_2\in J_{14,\Lambda,2}^{\Co_0}$ satisfies the following second-order MLDE
\begin{equation}\label{MLDEB2}
 \Big[H^{2}-\frac{E_4}{18}\Big]B_2=0.   
\end{equation}
We also find that the two Conway invariant holomorphic Jacobi forms of weight $14$ and index $3$, that is, $B_3$ and $\Psi_{14,3}$ satisfy the same third-order MLDE
\begin{equation}\label{MLDEB3}
\Big[H^3-\frac{E_4}{12}H+\frac{E_6}{72}\Big]\phi=0.    
\end{equation}
For any even positive-definite unimodular lattice $L$ and any fixed index $t\ge 2$, we can define the similar holomorphic Jacobi form $B_t$ of weight $\frac{1}{2}\mathrm{rank}(L)+2$ and index $t$ for $L$ which is invariant under the orthogonal group of $L$. We remark that the MLDE of $B_t$ does not depend on the lattice $L$. In fact, the MLDE of $B_t$ is completely determined by the MLDE of the modular form of weight $2$ involved in the definition of $B_t$. Indeed, Sakai's Jacobi forms $B_2$ and $B_3$ for root lattice $E_8$ (see \cite{Sak17}) satisfy the same MLDEs as \eqref{MLDEB2} and \eqref{MLDEB3} respectively.

The Conway invariant holomorphic Jacobi form $\Psi_{16,3}$ of weight $16$ and index $3$ satisfies the following fifth-order MLDE 
\begin{equation}
 \Big[H^{5}-\frac{5E_4}{12}H^{3}+\frac{5E_6}{24}H^{2}-\frac{5E_4^2}{144}H\Big]\Psi_{16,3}=0.   
\end{equation}
This MLDE is integrable. In fact, we find
\begin{equation}
\Big[H^{4}-\frac{5E_4}{12}H^{2}+\frac{5E_6}{72}H\Big]\Psi_{16,3}=\frac{88\Delta}{887} \big({85562484 (252 A_3+\Phi_{12,3})-3359 \Psi_{12,3}}\big).    
\end{equation}
Due to the singular weight property and the equality $H(\Delta \phi)=\Delta H(\phi)$, the $H$-image of the above right hand side is indeed zero.

The MLDEs for the products of two Jacobi forms usually have high orders. For example, we find that
$A_1\Phi_{12,2}$ satisfies the following MLDE of order ten
\begin{equation}
\begin{aligned}
&\Big[H^{10}+\mu_1E_4H^{8}+\mu_2E_6H^{7}+\mu_3E_4^2H^{6}+\mu_4E_4E_6 H^{5}+ (\mu_5E_6^2  +\mu_6 E_4^3)H^{4} \\
&+ \mu_7 E_4^2 E_6 H^{3}+ 
 (\mu_8 E_4 E_6^2  + \mu_9 E_4^4 )H^{2} +  (\mu_{10}E_6^3  +\mu_{11} E_4^3 E_6)H\Big](A_1\Phi_{12,2})=0
\end{aligned}    
\end{equation}
where
\begin{equation}
\begin{aligned}
&\mu_1=-\frac{25}{6},\quad \mu_{2}= \frac{235}{36},\quad\mu_{3}= -\frac{35}{16},\quad\mu_{4}= -\frac{805}{144},\quad\mu_{5}= \frac{7325}{1728},\quad\mu_{6}= \frac{1075}{216},\\
&\mu_{7}= -\frac{23275}{3456},\quad \mu_{8}= \frac{2525}{1728},\quad \mu_{9}= \frac{1475}{1296},\quad \mu_{10}= -\frac{25}{432},\quad \mu_{11}= -\frac{11675}{31104}.
\end{aligned}    
\end{equation}
To prove this MLDE, it is sufficient to verify the Fourier expansion up to $q^4$-term. We remark that the MLDEs for $A_1A_2$, $A_1B_2$ and $A_1HB_2$ have orders $13$, $14$ and $15$. In the next subsection, we will also encounter systems of MLDEs, which we also call modular linear relations.

\subsection{The product decomposition of Conway orbits}
In previous sections, we have determined the decompositions of $O_2\otimes O_2$, $O_2\otimes O_3$, $O_2\otimes O_4$ and $O_3\otimes O_3$. We now use the technique of Conway invariant Jacobi forms of index $3$ to determine more product decompositions of Conway orbits. As we mentioned in \S \ref{sec:phi12}, the $q^6$-term of $\Phi_{12,2}$ enables us to determine 
\begin{equation}
\begin{aligned}
O_{3}\otimes O_4+O_{2}\otimes O_5=&\,4194304 O_{2} + 1612875 O_{3} + 565248 O_{4} + 177400 O_{5}+ 48576 O_{6 a}\\
&+ 47104 O_{6 b} + 
 11177 O_{7} + 2048 O_{8 b}+ 2300 O_{8 c} + 276 O_{9 a}\\
 &+ 276 O_{9 b} + 24 O_{10 a} + 
 O_{11 a}    + 
 O_{11 b} + O_{12 b}  + O_{12 c}.
\end{aligned}    
\end{equation}
This allows us to make ansatz
\begin{equation}\label{OO2534}
O_2\otimes O_5=\sum_{O\in S_1} x_* O_{*} ,\qquad
O_3\otimes O_4= \sum_{O\in S_1} y_* O_{*},  
\end{equation}
where $S_1$ is the set of 16 orbits 
$$
O_{2}, O_{3} , O_{4} ,O_{5}
,O_{6 a} , O_{6 b} , O_{7}, O_{8 b}, O_{8 c}, O_{9 a}  , O_{9 b}, O_{10 a}, O_{11 a}, O_{11 b}, O_{12 b}, O_{12 c}.
$$
On the other hand,
the $q^5$-terms of $A_1A_2$, $A_1B_2$, $A_1HB_2$ and $A_1\Phi_{12,2}$ contain orbit products $O_{2}\otimes O_5$, $O_{2}\otimes O_{6a}$ and $O_{2}\otimes O_{6b}$. We also need to make ansatz for the last two of them. 
From bound (\ref{pqbound}), we know $O_2\otimes O_6 $ at most contains orbits of type from $2$ to $14$. 
Furthermore, the $q^7$-term of $\Phi_{12,2}$ enables us to determine the decomposition of $O_4\otimes O_4+2O_3\otimes O_5+2O_2\otimes (O_{6a}+O_{6b})$. 
From these two constraints we obtain the following most general ansatz
\begin{equation}
O_2\otimes O_{6a}=\sum_{O\in S_2} z^a_* O_{*} ,\qquad
O_2\otimes O_{6b}= \sum_{O\in S_2} z^b_* O_{*},    
\end{equation}
where $S_2$ is the set of $26$ orbits 
$O_{2}, O_{3}, O_{4}, O_{5}, O_{6 a}, O_{6 b}, O_{7}, O_{
 8 a}, O_{8 b}, O_{8 c}, O_{9 a}, O_{9 b}, O_{10 a}, O_{10 b}, O_{
 10 c}$, $O_{11 a}, O_{11 b}, O_{12 d}, O_{12 e}, O_{
 12 f}, O_{13 a}, O_{13 b}, O_{13 c}, O_{14 c}, O_{14 d}, O_{
 14 e}.$
Now we can compute $A_1A_2$, $A_1B_2$, $A_1HB_2$ and $A_1\Phi_{12,2}$ up to (and including) $q^5$-terms with $68$ parameters $x_*,z^a_*,z^b_*$ involved. We consider the modular linear relations among the following holomorphic Jacobi forms
\begin{equation}
 A_3,\Phi_{12,3},\Psi_{12,3},H^iB_3,H^i(A_1A_2),H^i(A_1B_2),H^i(A_1HB_2),H^i(A_1\Phi_{12,2}).   
\end{equation}
Due to MLDE \eqref{MLDEB3}, for $H^iB_3$ we only need to consider $i=0,1,2$.
Two modular linear relations at weight 28 have been given in \eqref{O2O4} and \eqref{psi123}. We find five more modular linear relations at weight $30$, among which we display two examples here
\begin{equation}\label{weight30one}
\begin{aligned}
0=&\ 4098 E_6A_1 A_2+E_4(5915 A_1 B_2 - 3415 H(A_1A_2) + 52 H(A_1\Phi_{12,2}))\\
&+8196 H^3(A_1A_2)- 64 H^3(A_1\Phi_{12,2}) - 8190 H^2(A_1B_2)  + 49140 H(A_1HB_2),
\end{aligned}    
\end{equation}
\begin{equation}\label{weight30two}
\begin{aligned}
0=&\,44287 E_6A_3 (16695 E_4^3 + 11689 E_6^2)-574776 A_1 (61425 E_4B_2 + 2 E_6 (14343 A_2 + 65 \Phi_{12,2}))\\
&-1296 \Delta E_6 (223767 \Phi_{12,3} + \Psi_{12,3})+21288 (66430E_4 E_6^2 B_3  - 81 E_4 (10245 H(A_1A_2)\\
&+ 64 H(A_1\Phi_{12,2})) - 
   398580 E_4^2 E_6 HB_3 + 2730 (2187 H(A_1HB_2) + 292 E_6^2 H^2B_3)).
\end{aligned}    
\end{equation}
These identities are established by checking their Fourier expansions up to $q^3$-terms. 
Then the $q^5$-terms of all these seven modular linear relations give many constraints and together determine $60$ of the $68$ parameters, but left with $8$ ones of the large orbits $x_{11a},x_{11b},z^a_{13 a},z^a_{13 b},z^a_{13 c},z^b_{14 c},z^b_{14 d},z^b_{14 e}$. It turns out that all these $8$ left parameters can be uniquely fixed by the condition $|O\otimes O|=|O|\times|O|$ and the requirement that all $x_*,z^a_*,z^b_*$ are non-negative integers. For example, $|O_2\otimes O_5|=|O_2|\times|O_5|$ implies 
$$
1273079808000x_{11a} + 25779866112000 x_{11b}=1273079808000, \quad \text{i.e.} \quad 4x_{11a} +81x_{11b}=4.
$$
Clearly, the only non-negative integer solution is $x_{11a}=1,x_{11b}=0$. By this process, we completely determine the product decompositions of $O_{2}\otimes O_5$, $O_{2}\otimes O_{6a}$, $O_{2}\otimes O_{6b}$ and by extension $O_3\otimes O_4$. The results are collected in Appendix \ref{app:orbitprod}.

Furthermore, as we mentioned in \S \ref{sec:phi12}, the $q^4$-term of $\Phi_{12,3}$ enables us to find
\begin{equation}
\begin{aligned}
O_{2}\otimes( 2 O_5-O_{6b}+O_{8a})=&\,89609 O_{2} + 151800 O_{3} + 93196 O_{4} + 
 38350 O_{5}  + 11868 O_{6 a}\\
 &+ 9484 O_{6 b} + 2804 O_{7}- 4600 O_{8 a} + 484 O_{8 b}+ 276 O_{8 c}+ 
 23 O_{9 a} \\
 &+ 
 66 O_{9 b} + 3 O_{10 a}- 43 O_{10 b} + 3 O_{12 b}  - 3 O_{12 d}+O_{18a}.
\end{aligned}    
\end{equation}
This allows us to also fix the decomposition of $O_2\otimes O_{8a}$, which is presented in Appendix \ref{app:orbitprod}.

In conclusion, we have proved the decompositions for nine products of Conway orbits which are formulated in Appendix \ref{app:orbitprod}.

To determine the $q^5$-terms of holormorphic generators $\Psi_{14,3}$ and $\Psi_{16,3}$ and more orbit products decomposition, we need to compute $A_3,B_3,A_1A_2,A_1\Phi_{12,2},A_1B_2,A_1HB_2$ to $q^6$-terms. This of course cannot be achieved completely due to the lack of classification on Conway orbits of types $17$ and $18$. However, we can introduce two pseudo-orbits $O_{17',49597544448p},O_{18b',93053764668p}$ besides $O_{18a}$, in the sense that we do not distinguish the orbits of type $17$ and the orbits of type $18$ except $O_{18a}$. This is legitimate for our current purpose, since the orbits of the same type are uniformly transformed under the differential operators $H$. Therefore, our following computations and results are rigorous for all orbits of type smaller than $17$ and also $O_{18a}$. Note that the two pseudo-orbits appearing in the orbit product decomposition do not necessarily have integer coefficients. 

With the two pseudo-orbits, we can compute $A_3$ and $B_3$ up to their $q^6$-terms.
The $q^6$-terms of $A_1A_2$, $A_1\Phi_{12,2}$, $A_1B_2,A_1HB_2$ involve the orbit products $O_2\otimes O_{7},$ $O_2\otimes O_{8b},$ $O_2\otimes O_{8c},$ $O_3\otimes O_{5},$ $O_3\otimes O_{6a}$, $O_3\otimes O_{6b}$ and $O_4\otimes O_4$. To write ansatz for the decompositions of these orbit products, we first make use of the Fourier expansion of $\Phi_{12,2}$. The $q^7$-term of $\Phi_{12,2}$ gives us
\begin{equation}\label{eq:reducible01}
\begin{aligned}
&O_4\otimes O_4+2O_3\otimes O_5\\
=&\,398034000+45742400 O_{2}+20165112 O_{3}+8269144 O_{4}\!+3135550 O_{5}+1087440 O_{6 a}\qquad\\
&+1095120 O_{6 b}+339000 O_{7}+187358 O_{8 a}+92614 O_{8 b}+89606 O_{8 c}+21298 O_{9 a}\\
&+21384 O_{9 b}+3940 O_{10 a}+4512 O_{10 b}+4400 O_{10 c}+548 O_{11 a}+536 O_{11 b}+42 O_{12 d}\\
&+46 O_{12 e}+48 O_{12 f}+2 O_{13 b}+2 O_{13 c}+2 O_{14 d}+2 O_{14 e}+O_{16 a}.
\end{aligned}    
\end{equation}
This enables us to make ansatz for $O_4\otimes O_4$ and $O_3\otimes O_5$ using the orbits appearing in the right hand side of the above decomposition.
The $q^8$-term of $\Phi_{12,2}$ gives us
\begin{equation}\label{eq:reducible02}
\begin{aligned}
&O_4\otimes O_5+O_2\otimes O_7+O_3\otimes O_{6a}+O_3\otimes O_{6b}\\
=&\,99532800 O_{2} + 48897678 O_{3} + 22892544 O_{4} + 
 10131156 O_{5}  + 4200352 O_{6 a}+ 4194304 O_{6 b}\\
 &+ 1612875 O_{7}+ 
 565248 O_{8 b} + 570078 O_{8 c}+
 177399 O_{9 a}+ 177399 O_{9 b}  + 48576 O_{10 a} \\
 &+ 47104 O_{10 b}+ 47104 O_{10 c}+ 11178 O_{11 a} + 
 11178 O_{11 b} + 2300 O_{12 b} + 
 2300 O_{12 c} + 
 2048 O_{12 d}\\
& + 2048 O_{12 e} + 
 2048 O_{12 f} + 
 276 O_{13 a} + 276 O_{13 b}+ 276 O_{13 c}   + 24 O_{14 a}+ 24 O_{14 b} + O_{15 a} \\
 &+ 
 O_{15 b}    + O_{15 c}  + O_{15 d} + O_{15 e}+ O_{16 b} + O_{16 c}.
\end{aligned}    
\end{equation}
This enables us to make ansatz for $O_2\otimes O_7$, $O_3\otimes O_{6a}$ and $O_3\otimes O_{6b}$. Note that
the $q^5$-term of $\Phi_{12,3}$ relates $O_2\otimes O_7$ and $O_3\otimes O_{6b}$ by
\begin{equation}\label{eq:reducible03}
\begin{aligned}
&O_2\otimes O_7-O_3\otimes O_{6b}\\
=&-64722 O_{3} + 32481 O_{5} + 30912 O_{6 a}+ 
 35840 O_{6 b}  + 17875 O_{7}
 + 7680 O_{8 b} + 7128 O_{8 c}\\
 &+ 
 2300 O_{9 a}+ 2664 O_{9 b}+ 660 O_{10 a} + 704 O_{10 c}
 + 130 O_{11 b} - 
 275 O_{12 b}+ 
 23 O_{12 c}+ 24 O_{12 f}\\
 &- 23 O_{13 a} + O_{13 b}    + O_{13 c} + O_{15 a}- O_{15 c}.
\end{aligned}    
\end{equation}
For $O_2\otimes O_{8b}$ and $O_2\otimes O_{8c}$, by bound \eqref{pqbound} they at most contain orbits of type from 2 to 18. 
With the ansatz for all these seven orbit products, we can finally compute the $q^6$-terms of $A_1A_2,A_1\Phi_{12,2},A_1B_2$ and $A_1HB_2$ with parameters.  

Consider the $q^6$-terms of the two weight $28$ modular linear relations \eqref{O2O4} and \eqref{psi123} and the five weight $30$ modular linear relations including \eqref{weight30one} and \eqref{weight30two}. In these relations, when $\Phi_{12,3}$ and $\Psi_{12,3}$ appear, they are always multiplied by $\Delta$ (e.g. see \eqref{weight30two}). 
Therefore, we can compute the Fourier expansion of each monomial in these relations up to $O(q^7)$. Then the $q^6$-terms of all the modular linear relations give a large number of constraints on the ansatz. Together with the requirement on the non-negative integrity of the decomposition coefficients, we are able to determine the following combinations of orbit products:
\begin{equation}\label{eq:reducible04}
\begin{aligned}
&O_4\otimes O_4 + O_3\otimes O_{6a} + O_2\otimes O_{8b}\\
=&\,398034000 + 36936000 O_{2} + 17512200 O_{3} + 8238232 O_{4} + 
 3653100 O_{5}  + 1515010 O_{6 a}\\
 &+ 
 1518264 O_{6 b} + 581900 O_{7}+ 93150 O_{8 a}+ 204030 O_{8 b} + 202500 O_{8 c}+ 
 63756 O_{9 a}\\
 &+ 64152 O_{9 b} + 17525 O_{10 a} + 
 16192 O_{10 b}  + 17525 O_{10 c} + 4050 O_{11 a}   + 4032 O_{11 b} + 
 759 O_{12 c}\\
 &+ 890 O_{12 d}+ 
 758 O_{12 e}+ 728 O_{12 f}+ 100 O_{13 b} + 100 O_{13 c} + 9 O_{14 b}  + 9 O_{14 e} + 
 O_{15 b}+ O_{15 e}\\
 &  + O_{16 a}+ O_{16 b} + O_{16 d}.
\end{aligned}    
\end{equation}

\begin{equation}\label{eq:reducible05}
\begin{aligned}
&O_3\otimes O_5 + O_2\otimes O_{7}\\
=&\,6476800 O_{2} + 3945834 O_{3} + 1884160 O_{4} + 802406 O_{5}+ 310224 O_{6 a}+ 
 311040 O_{6 b}\\
& + 
 107625 O_{7}  + 47104 O_{8 a}+ 32768 O_{8 b} + 31681 O_{8 c}+ 8395 O_{9 a}  + 8505 O_{9 b}+
 1782 O_{10 a}\phantom{iii} \\
&+ 2048 O_{10 b} + 1804 O_{10 c}+ 274 O_{11 a}    + 
 286 O_{11 b}+ 
 23 O_{12 c}   + 24 O_{12 e} + 36 O_{12 f} + 2 O_{13 b}\\
& + 2 O_{13 c}+ O_{14 a}+ O_{14 d} + O_{15 a}.
\end{aligned}    
\end{equation}

\begin{equation}\label{eq:reducible06}
\begin{aligned}
&O_3\otimes O_{6b}+ O_2\otimes O_{8c}\\
=&\,76452 O_{3} + 47104 O_{4} + 26875 O_{5} + 
 12696 O_{6 a}+ 
 11264 O_{6 b} + 5402 O_{7} + 47104 O_{8 a}\phantom{iiiiiiii,}\\
 &+ 2048 O_{8 b}+ 4325 O_{8 c} + 1012 O_{9 a}
 + 561 O_{9 b} + 198 O_{10 a}  + 2048 O_{10 b}+ 
 200 O_{10 c}\\
 &+
 274 O_{11 a}+ 
 36 O_{11 b}+ 552 O_{12 a}  + 550 O_{12 b} + 23 O_{12 c} + 24 O_{12 e}+ 46 O_{13 a}  + O_{13 b}\\
 &    + O_{13 c}+ 2 O_{14 a} + 
 2 O_{14 d}    + 3 O_{15 c}  +c_{17'}O_{17'},
\end{aligned}    
\end{equation}
where $c_{17'}=1/350979$. This $c_{17'}$ suggests that at type $17$ there probably exists an orbit $O_{17,141312p}$ such that its coefficient in the above decomposition is $1$. 

Unfortunately, we cannot separate the above decomposition of orbit products. There may be some opportunity to use Conway invariant Jacobi forms of index $4$ to separate these combinations. The above \eqref{eq:reducible04}, \eqref{eq:reducible05} and \eqref{eq:reducible06} appear in pairs in the Fourier expansion of any product of $A_1$ and a Conway invariant Jacobi form of index $2$, since $O_4$ is conjugate to $O_{6a}$ and $O_{8b}$, and $O_5$ is conjugate to $O_7$, and $O_{6b}$ is conjugate to $O_{8c}$, in the sense of modulo $2\Lambda$ (see \S \ref{subsec:conjugate classes}). Therefore, the above combinations are sufficient for us to fix $\Psi_{14,3}$ and $\Psi_{16,3}$ up to $O(q^6)$, and $A_1A_2,A_1\Phi_{12,2},A_1B_2,A_1HB_2$ up to $O(q^7)$. The Fourier expansions of these holomorphic Jacobi forms are very useful in determining the conjugate relations among Conway orbits modulo $2\Lambda$ and $3\Lambda$, which will be discussed in \S \ref{subsec:conjugate classes}.

\subsection{Conway invariant Jacobi forms of singular weight and non-trivial character}
Let $v_\eta$ be the multiplier system of the Dedekind eta function 
$$
v_{\eta}(A) = \frac{\eta(A\cdot \tau)}{\eta(\tau) \sqrt{c\tau+d}}, \quad \eta(\tau)=q^{\frac{1}{24}}\prod_{n=1}^\infty (1-q^n), \quad A=\left( \begin{array}{cc}
a & b \\ 
c & d
\end{array} \right)   \in \SL_2(\ZZ).
$$
We classify singular-weight Conway invariant holomorphic Jacobi forms of  index $t\leq 3$ and character $v_{\eta}^d$, where $d$ is an even integer which only depends on the conjugate class modulo $24$. Through the theta decomposition \cite{Gri94}, such forms are one-to-one corresponding to invariants of the Weil representation of $\SL_2(\ZZ)$ attached to $\Lambda / t\Lambda$ which are invariant under $\Co_0$. Similar to the case of trivial character, such a form takes the Fourier expansion of the form
\begin{equation}
\Phi_{12,t;a}(\tau,\mathfrak{z})=\sum_{n\in \frac{d}{24} + \ZZ} \sum_{\substack{\ell \in \Lambda \\ (\ell,\ell)=2nt}} f(n,\ell)q^n\zeta^{\ell},
\end{equation}
where $a=24/(24,d)$ indicates the order of the character $v_{\eta}^d$. 

Clearly, the singular-weight Conway invariant holomorphic Jacobi forms of index $1$ and non-trivial character do not exist. When index $t=2$, this type of form does not exist either. Otherwise, if such a form exists, then $d$ can only be 12, and the Fourier expansion has to be
$$
\Phi_{12,2;2}(\tau,\mathfrak{z})=q^{\frac{3}{2}}O_3 + O(q^{\frac{5}{2}}),
$$
which contradicts the reduction $\Phi_{12,2;2}(\tau,0)=0$ because $\Phi_{12,2;2}/\eta^{36}\in J_{-6, \Lambda,2}^{\w, \Co_0}$. 

When $t=3$, this type of Jacobi form does exist. We construct a Conway invariant holomorphic Jacobi form of weight $16$ and index $3$
$$
F_{16,3} =\frac{E_4(3359 \Psi_{12,3}-85562484 (252 A_3 + \Phi_{12,3}))}{7311718400}+\frac{6251063}{920}(HB_3-H \Psi_{14,3})-\frac{3 \Psi_{16,3}}{90675200},
$$
and find that
\begin{equation}\label{eq:special}
\begin{aligned}
\Phi_{12,3;3}:=&\, \eta^{-8} F_{16,3} \\
=&\,\frac{q^{\frac23}}{8096}\Big(8096O_2+77O_5q+(8096 O_{8 a}-4 O_{8 b}+77 O_{8 c}) q^2+(77 O_{11 a}-4 O_{11 b}) q^3\\
&+(77 O_{14 a}-4 O_{14 b}+8096 O_{14 c}+77 O_{14 d}-4 O_{14 e}) q^4+O(q^5)\Big).
\end{aligned}    
\end{equation}
By the product decomposition in Appendix \ref{app:orbitprod}, we are able to calculate the $q^n$-terms in the above bracket for $n\leq 3$. This implies that $\Phi_{12,3;3}$ is a singular-weight Conway invariant holomorphic Jacobi form of index $3$ and character $v_{\eta}^{16}$. We then determine the $q^4$-term by solving the linear equation
$$
x_a|O_{14a}| + x_b|O_{14b}| + x_c|O_{14c}| + x_d|O_{14d}| + x_e|O_{14e}| = 11420136000,
$$
where these coefficients $x_*$ can only be $8096$, $77$ or $-4$, and the constant on the right hand side is determined by the $q^{14/3}$-term of the reduction $\Phi_{12,3;3}(\tau,0)$. Besides, the $q^4$-term can also be computed by using the $q^5$-terms of holomorphic generators which are determined at the end of the previous subsection. We have checked that the two types of calculation are consistent. We then arrive at the following classification.

\begin{proposition}
The form $\Phi_{12,3;3}$ is the unique (up to scalar) singular-weight Conway invariant holomorphic Jacobi form of index $t\leq 3$ and non-trivial character. 
\end{proposition}
\begin{proof}
It suffices to prove that there is no other singular Jacobi form of index $3$ and character $v_{\eta}^{16}$ which is linearly independent of $\Phi_{12,3;3}$, and there is no singular Jacobi form of index $3$ and character $v_{\eta}^{8}$. Suppose that the second singular-weight Jacobi form of index $3$ and character $v_{\eta}^{16}$ exists. Then its Fourier expansion will have the form 
$$
\Psi_{12,3;3} = q^{\frac{5}{3}}O_5 + O(q^{\frac{8}{3}}) \quad \text{or}
\quad q^{\frac{8}{3}}O_{8b} + O(q^{\frac{11}{3}}),
$$
which contradicts the reduction $\Psi_{12,3;3}(\tau,0)=0$. Suppose that the singular-weight Jacobi form of index $3$ and character $v_{\eta}^{8}$ exists. Then its Fourier expansion takes the form
$$
\Psi_{12,3;3} =q^{\frac{4}{3}}O_{4} + O(q^{\frac{7}{3}}) \quad \text{or}
\quad q^{\frac{7}{3}}O_{7} + O(q^{\frac{10}{3}}),
$$
which contradicts the reduction $\Psi_{12,3;3}(\tau,0)=0$ again. We then complete the proof. 
\end{proof}

\subsection{Conjugate classes of orbits modulo \texorpdfstring{$t\Lambda$}{}}\label{subsec:conjugate classes}
According to Lemma \ref{lem:coefficients}, in the Fourier expansion of a Conway invariant weak Jacobi form of index $t$, the coefficients $f(n_1,\ell_1)$ and $f(n_2,\ell_2)$ are equal if $2n_1t-(\ell_1,\ell_1)=2n_2t-(\ell_2,\ell_2)$ and $\ell_1 - \ell_2 \in t\Lambda$. From this basic fact, we can determine whether or not two orbits of $\Lambda / \Co_0$ are conjugate modulo $t\Lambda$ by observing the Fourier expansions of some nice Jacobi forms. We first consider the case of $t=2$. The orbits $O_0$, $O_2$, $O_3$ and $O_4$ form a minimal length representative system of $(\Lambda / \Co_0) / 2\Lambda$. Clearly, every orbit of odd type is conjugate to $O_3$ modulo $2\Lambda$. The following result describes the conjugate classes of orbits of even type $x\leq 16$, which can be found in \cite[Page 181]{ATLAS}. We give it a simple proof based on Jacobi forms.

\begin{proposition}\label{prop:modulo2}
For orbits of even type $x\leq 16$, the modulo $2\Lambda$ conjugate classes are as follows
\begin{align*}
&O_0:\quad O_{8a},\; O_{12a},\; O_{16a};&\\
&O_2:\quad O_{6b},\; O_{8c},\; O_{10b}, \; O_{10c} ,\; O_{12b} ,\; O_{12c} ,\; O_{14c} ,\; O_{14d} ,\; O_{14e} ,\; O_{16b} ,\; O_{16c};&\\
&O_4:\quad O_{6a},\; O_{8b},\; O_{10a}, \; O_{12d} ,\; O_{12e} ,\; O_{12f} ,\; O_{14a} ,\; O_{14b} ,\; O_{16d} ,\; O_{16e} ,\; O_{16f} ,\; O_{16g}.&
\end{align*}
\end{proposition}
\begin{proof}
The proof follows from the Fourier expansion \eqref{c2order8} of $\Phi_{12,2}$ (does not rely on the conjugate relations modulo $2\Lambda$), where the terms $q^0 O_0$, $q O_2$ and $q^2 O_4$ have distinct coefficients $24$, $1$, $0$. 
\end{proof}

We now consider the case of $t=3$. We prove the following theorem. 
\begin{proposition}
For orbits of type $x \leq 16$, the modulo $3\Lambda$ conjugate classes are as follows
\begin{align*}
&O_0:\quad\,\, \emptyset; & \\
&O_2:\quad \,\,  O_{8a},\; O_{14c};&\\
&O_3:\quad \,\,  O_{12a},\; O_{12b}, \; O_{15a};&\\
&O_4:\quad \,\,  O_{10b},\; O_{13a}, \; O_{16a}, \; O_{16b},\; O_{16d};& \\
&O_5:\quad   \,\,O_{8c},\; O_{11a}, \; O_{14a}, \; O_{14d};&\\
&O_{6a}:\quad   O_{9a},\; O_{12c}, \; O_{12e}, \; O_{15b}, \; O_{15e};&\\
&O_{6b}:\quad   O_{12d},\; O_{15c};&\\
&O_{7}:\quad \,\,  O_{10a},\; O_{10c}, \; O_{13b}, \; O_{13c}, \; O_{16c},\; O_{16e},\; O_{16f},\; O_{16g};&\\
&O_{8b}:\quad   O_{11b},\; O_{14b}, \; O_{14e};&\\
&O_{9b}:\quad   O_{12f},\; O_{15d}.&
\end{align*}
\end{proposition}
\begin{proof}
We first determine the conjugate classes of orbits of type $3x$. We derive the conjugate classes of $O_0$, $O_3$ and $O_{6b}$ from the Fourier expansion of $\Phi_{12,3}$ given in \eqref{c2order8}, because the terms $q^0O_0$, $qO_3$ and $q^2O_{6b}$ have distinct coefficients $-252$, $1$ and $-1$. We further determine the conjugate classes of $O_{6a}$ and $O_{9b}$ from the Fourier expansion of $\Psi_{12,3}$ given in \eqref{psi123}, because the terms $q^2O_{6a}$ and $q^3O_{9b}$ have distinct coefficients $-644$ and $243$. 

We then consider the conjugate classes of orbits of type $3x+2$. We read the information from the Fourier expansion of $\Phi_{12,3;3}$ given in \eqref{eq:special}, where the terms $q^{2/3}O_2$, $q^{5/3}O_5$ and $q^{8/3}O_{8b}$ have distinct coefficients $1$, $77/8096$ and $-4/8096$.

We observe the conjugate classes of orbits of type $3x+1$ from the Fourier expansion of $A_1A_2$, where the terms $q^2O_4$ and $q^3O_7$ have distinct  coefficients $1/2049$ and $0$.
\end{proof}

\subsection{Pullbacks of Conway invariant Jacobi forms to Jacobi forms of Eichler--Zaiger} 
We study the pullback of Conway invariant Jacobi forms to classical Jacobi forms of Eichler--Zaiger \cite{EZ85} along a Leech vector. The Fourier coefficients of these pullback forms are very useful to characterize the intersection of Leech vectors, which are not easy to obtain by brutal computations.

Let $v_a$ be a Leech vector of type $a$ and $z\in \CC$. For any $f_{k,t} \in J_{k,\Lambda,t}^{\w, \Co_0}$, the pullback $f_{k,t}(\tau, zv_a)$ along $v_a$ is a weak Jacobi form of weight $k$ and index $at$ in the sense of Eichler--Zagier \cite{EZ85}. 
The Conway invariance yields that $f_{k,t}(\tau, zv_a)$ depends only on the $\Co_0$-orbit of $v_a$. We express $f_{k,t}(\tau, zv_a)$ in terms of the standard generators 
\begin{equation}
\phi_0(\tau,z)=u^{\pm 1}+10+O(q), \quad 
\phi_{-2}(\tau,z)=u^{\pm 1}-2+O(q)    
\end{equation}
of the polynomial algebra of classical weak Jacobi forms of even weight and integral index over $\CC[E_4, E_6]$, where 
$$
u^{\pm n}:=e^{2\pi inz} + e^{-2\pi inz}, \quad \text{for $n\in\NN$}.
$$
For two Leech vectors $x$ and $y$, we define their \textit{intersection number} as the inner product $(x,y)$. 
By means of the expression of $f_{k,t}(\tau, zv_a)$, we calculate the pullbacks of Conway orbits
\begin{equation}\label{eq:pullback}
\orb(r)(zv_a)=\sum_{v\in \Co_0\cdot r} e^{2\pi i (v, v_a)z},  
\end{equation}
which describe the intersection numbers of Leech vectors. We next consider some examples.

We first determine the pullback of Conway invariant Jacobi form $A_1$ along $v_2$. From the obvious fact $A_1(\tau,zv_2)=1+O(q^2)$, we derive
\begin{equation}
\begin{aligned}
A_1(\tau, zv_2)=&\,\frac{1}{1728}\Big((7 E_4^3+5E_6^2)\phi_{0}^2-24E_4^2E_6\phi_0\phi_{-2}+3E_4(E_4^3+3E_6^2)\phi_{-2}^2\Big)\\
=&\,1+q^2 \big( u^{\pm 4}+4600  u^{\pm 2} +47104u^{\pm 1}+93150\big)\\
&+q^3 \big(47104  u^{\pm 3}+953856  u^{\pm 2}+4147200  u^{\pm 1}+6476800\big)+ O(q^4).
\end{aligned}    
\end{equation}
The above $q^2$-term characterizes the intersection between $O_2$ and $v_2$. More precisely, among all $196560$ Leech vectors of type $2$ there are exactly $1$, $4600$, $47104$ vectors whose intersection numbers with $v_2$ are respectively $4$, $2$, $1$, and there are exactly $93150$ vectors orthogonal to $v_2$. Similarly, we read the intersections between $O_3$, $O_4$, $O_5$, $O_7$ and $v_2$ from the $q^3$-, $q^4$-, $q^5$-, $q^7$-terms of the above pullback respectively. However, we cannot determine from $A_1(\tau, zv_2)$ itself the intersection between other Conway orbits and $v_2$,  because there are multiple orbits of the same type. 

We further consider the pullback of $\Phi_{12,2}$ along $v_2$. We find that its $q^{0}$- and $q^{2}$-terms are sufficient to determine
\begin{equation*}
\begin{aligned}
\Phi_{12,2}(\tau, z v_2)=&\, \frac{1}{82944}\Big(\phi_{0}^4 \left(511 E_4^3-415 E_6^2\right)-384 E_4^2 E_6 \phi_{0}^3 \phi_{-2}+2 E_4 \phi_0^2 \phi_{-2}^2 \left(779 E_6^2-491 E_4^3\right)\\
&\, +48 E_6 \phi_0 \phi_{-2}^3 \left(7 E_4^3-15 E_6^2\right)+3 E_4^2 \phi_{-2}^4 \left(61 E_4^3-29 E_6^2\right)\Big)\\
=&\,24+q \big( u^{\pm 4} +4600 u^{\pm 2} +47104 u^{\pm 1}+93150\big) +q^3 \big(4600  u^{\pm 6}+2049300  u^{\pm 4}\\
&+12953600  u^{\pm 3}+51791400  u^{\pm 2}+95385600  u^{\pm 1}+127719000\big)+O(q^4).
\end{aligned}    
\end{equation*}
The intersection between $O_{6b}$ and $v_2$ corresponds to the above $q^3$-term. We note that the intersection between $O_{8a}$ and $v_2$ follows from the intersection between the primitive orbit $O_2$ and $v_2$. Combining the two pullbacks together, we further determine the intersections between $O_{6a}$, $O_{8a}$, $O_{8b}$, $O_{8c}$, $O_{10a}$ and $v_2$. The intersections between $O_{9a}$, $O_{9b}$ and $v_2$ can be determined by further computing the pullback of $\Psi_{14,3}$, of which we omit the expression here. We also calculate the pullbacks of $A_1$ and $\Phi_{12,2}$ along $v_3$ and $v_4$, and use these Fourier expansions to determine the intersections of orbits of type less than $9$ with $v_3$ and $v_4$. The exponential polynomials to express these intersections are formulated in Appendix \ref{app:orbitinter}.

Certainly, by calculating the pullbacks of more Jacobi forms along $v_a$, we can determine more intersections between Conway orbits and $v_a$. For any $v,u\in \Lambda/\Co_0$, we have the following relation between pullbacks of Conway orbits defined by \eqref{eq:pullback}
$$
|\orb(u)| \cdot \orb(v)(zu)=|\orb(v)| \cdot \orb(u)(zv).
$$
Thus the known intersections are useful to calculate the pullbacks of Conway invariant Jacobi forms of large index along Leech vectors of large type.

\section{Conway invariant Jacobi forms of higher index}\label{sec:high index}
In this section we discuss Conway invariant Jacobi forms of index larger than $3$ and propose some open questions. We first determine the minimal-norm representatives of $\Lambda / 4\Lambda$. The following theorem is derived from Borcherds' thesis \cite{Bor85}.
\begin{theorem}\label{th:system4}
Representatives of $\Lambda / 4\Lambda$ of minimal norm may be found among vectors of types up to $16$, according to the weighted equality
\begin{equation}\label{eq:system4}
\begin{aligned}
4^{24}=\,&1+|O_{2}| + |O_{3}| + |O_{4}| + |O_{5}| + |O_{6a}| + |O_{6b}|+|O_{7}|+\frac12|O_{8a}|+ |O_{8b}|  + |O_{8c}| 
\\
&+|O_{9a}|+|O_{9b}|+|O_{10a}|+\frac12 |O_{10b}|+ |O_{10c}|+\frac12|O_{11a}|+|O_{11b}|+\frac12|O_{12a}|\\
&+|O_{12c}|+\frac14|O_{12d}|+\frac12|O_{12e}| +|O_{12f}|+\frac12|O_{13b}|+\frac13|O_{13c}|+\frac16|O_{14b}| \\
&+\frac14|O_{14d}|+\frac18|O_{14e}|+\frac{1}{14}(|O_{15b}|+|O_{15d}|)+ \frac{1}{48}|O_{16a}| +\frac{1}{32}|O_{16e}|.
\end{aligned}    
\end{equation}
In the above equality, the weights are the $4$-weights of vectors in the orbits (see Definition \ref{def:weights}). 
Moreover, we have the following facts.
\begin{itemize}
    \item[(i)] The orbits $O_{15b}$ and $O_{15d}$ are conjugate modulo $4\Lambda$, but they are not conjugate to other orbits of type $15$.
    \item[(ii)] Every orbit $O_x$ in \eqref{eq:system4} not of type $15$ is not conjugate to any orbit $O_y$ of type $y\leq x$ modulo $4\Lambda$. Thus the $4$-weight of $O_x$ coincides with the $4$-weight of any vector in $O_x$. 
    \item[(iii)] The $4$-weights of $O_{15b}$ and $O_{15d}$ are respectively $2$ and $12$, which are different from the $4$-weight of any vector in $O_{15b}$ or $O_{15d}$, i.e. $14$. 
\end{itemize}
\end{theorem}
\begin{proof}
We know from Borcherds' thesis that the $\Co_0$-orbits of $\Lambda / 4\Lambda$ correspond one-to-one to the non-empty Dynkin diagrams related to orbits of vectors in $\mathrm{II}_{25,1}$ of height $4$, and the size of the Dynkin diagram describes the $4$-weight of any vector in the orbit, that is the number of Leech vectors of type $x$ which are conjugate (modulo $4\Lambda$) to a given vector of type $x$ (see Definition \ref{def:weights} for notions). We refer to \cite[Figure 2, Page 91]{Bor85} for details. More precisely, the minimal-norm representative system of $\Lambda / 4\Lambda$ consists of the following orbits. 
\begin{enumerate}
    \item All orbits of type $x\leq 11$. For these orbits, their $4$-weights coincide with the $4$-weights of vectors. There are unique orbits of types $8$, $10$ and $11$ which have $4$-weight $2$. The other orbits have $4$-weight $1$.
    \item Five orbits of type $12$, whose vectors have $4$-weights $1$, $1$, $2$, $2$ and $4$.
    \item Two orbits of type $13$, whose vectors have $4$-weights $2$ and $3$.
    \item Three orbits of type $14$, whose vectors have $4$-weights $4$, $6$ and $8$.
    \item One orbit of type $15$, whose any vector has $4$-weight $14$.
    \item Two orbits of type $16$, whose vectors have $4$-weights $32$ and $48$.
\end{enumerate}
   
Firstly, it is easy to determine the orbits of type $x\leq 11$ and weight $2$ by observing their shapes in ATLAS \cite{ATLAS}. For example, $O_{11a}$ has shape $u_3+2v_2$, and then $u_3+2v_2\sim_4 u_3-2v_2$, which implies that $O_{11a}$ has $4$-weight $2$ and thus $O_{11b}$ has $4$-weight $1$. 

For orbits of type $12$, we see from their shapes that $O_{12b}\sim_4 O_{8c}$. Thus any two of the remaining $5$ orbits of type $12$ are not conjugate modulo $4\Lambda$, and they are the orbits of minimal length in (2). Clearly, the $4$-weight of the non-primitive orbit $O_{12a}$ equals the $2$-weight of $O_3$ which is $2$. Then the other $4$ orbits must have $4$-weights $1$, $1$, $2$ and $4$. We need to consider $12$ possibilities in (2).

Again, we see from their shapes that $O_{13a}\sim_4 O_{9a}$. We cannot determine the weights of the other two orbits of type $13$. Thus we need to consider $2$ possibilities in (3).

There are $5$ orbits of type $14$. By the shapes in the ATLAS, $O_{14a}\sim_4 O_{10a}$, $O_{14c}\sim_4 O_{6b}$. Thus the other $3$ orbits are the orbits in (4) with $4$-weights $4$, $6$ and $8$. There are $6$ possibilities in (4). 

We have $O_{15a}\sim_4 O_7$ by their shapes. One or multiple of $O_{15b}$, $O_{15c}$, $O_{15d}$ and $O_{15e}$ will be the orbits of minimal length modulo $4\Lambda$ whose vectors have $4$-weight $14$.  Thus we need to consider all $2^4=16$ possibilities in (5). 

We first conclude from the shapes that $O_{16b}\sim_4 O_{12c}$. Since $O_4$ has minimal norm modulo $2\Lambda$, the non-primitive orbit $O_{16a}=\orb(2v_4)$ has minimal norm modulo $4\Lambda$. Moreover, the $4$-weight of $O_{16a}$ equals the $2$-weight of $O_{4}$ which is $48$. Clearly, if $v\sim_4 u$ then $v\sim_2 u$.  By Proposition \ref{prop:modulo2}, $O_{16c} \not\sim_4 O_{16a}$, $O_{16d}$, $O_{16e}$, $O_{16f}$, $O_{16g}$.  We cannot distinguish the orbits of type $16$ and $4$-weight $32$, which may be either $O_{16c}$, or one or multiple of $O_{16d}$, $O_{16e}$, $O_{16f}$ and $O_{16g}$. Thus we need to consider $1+2^4=17$ possibilities in (6). 

In total, we need to check $12\times 2\times 6 \times 16 \times 17 = 39168$ possibilities. By scanning the weighted equality of form \eqref{eq:system4}, we find that there is a unique solution which is exactly \eqref{eq:system4}. The properties (1)--(4) follow from this equality and the above discussion. 
\end{proof}

The above theorem yields that $O_{16c}$ is conjugate to some orbit of smaller type modulo $4\Lambda$. By Proposition \ref{prop:modulo2}, $O_{16c}$ is conjugate to $O_{8c}$ or $O_{12c}$ modulo $4\Lambda$. We see from the Fourier expansion of $\Phi_{12,4}$ given in \eqref{c2order8} that $O_{16c}\not\sim_4 O_{8c}$.  Thus $O_{16c}\sim_4 O_{12c}\sim_4 O_{16b}$. 

The above theorem yields the following facts on Conway invariant Jacobi forms of index $4$. 

\begin{proposition}
The rank of $J_{*,\Lambda,4}^{\w, \Co_0}$ is $31$. The basic Conway orbits of index $4$ are the $32$ orbits in \eqref{eq:system4}.
The number $\delta_4$ defined in Lemma \ref{lem:weak-holo} is $85$.
\end{proposition}

Besides, we conclude from $\Delta^4 J_{*,\Lambda,4}^{\w, \Co_0} \subset J_{*,\Lambda,4}^{\Co_0}$ that $J_{k,\Lambda,4}^{\w,\Co_0}=\{0\}$ for $k<-36$.
We check that $A_4$, $\Phi_{12,4}$, $A_1(\tau, 2\mathfrak{z})$, $\Phi_{12,2}|T_{-}(2)$ are linearly independent. Thus $\dim J_{12,\Lambda,4}^{ \Co_0} \geq 4$. A similar argument to Lemma \ref{lem:singular3} gives upper bound $\dim J_{12,\Lambda,4}^{ \Co_0} \leq 9$. We do not know the exact dimension. We note that $\dim J_{12,\Lambda,6}^{\Co_0} \geq 5$, because the following five Jacobi forms are linearly independent.
$$
A_6, \quad \Phi_{12,6}, \quad \Phi_{12,2}|T_{-}(3), \quad \Phi_{12,3}|T_{-}(2), \quad \Psi_{12,3}|T_{-}(2).
$$

By Theorem \ref{MTH}, there exist Conway invariant weak Jacobi forms of weight $4$ and index $2$ whose $q^0$-terms are a single orbit $O_0$, $O_2$, $O_3$ or $O_4$. Similarly, there exist Conway invariant weak Jacobi forms of weight $4$ and index $3$ whose $q^0$-terms are a single orbit $O_0$, $O_2$, $O_3$, $O_4$, $O_5$, $O_{6a}$, $O_{6b}$, $O_7$, $O_{8b}$ or $O_{9b}$. Therefore, each of the following orbit products
$$
\{O_0, O_2, O_3, O_4 \} \otimes \{O_0, O_2, O_3, O_4, O_5, O_{6a}, O_{6b}, O_7, O_{8b}, O_{9b} \}
$$
can be regarded as the $q^0$-term of a Conway invariant weak Jacobi form of weight $8$ and index $5$. Thus every orbit appearing in the decompositions of these products has minimal norm modulo $5\Lambda$, and if two orbits of the same type have distinct coefficients in a given product decomposition, then they are not conjugate modulo $5\Lambda$. In this way, from the decompositions of orbit products determined in this paper we derive that all orbits of type $x\leq 15$, $O_{16a}$, $O_{16b}$, $O_{16c}$, $O_{16d}$ and $O_{16e}$ have minimal norm modulo $5\Lambda$, and any two of them are not conjugate modulo $5\Lambda$. To prove this claim, we also use the simple fact that if $O_x$ has minimal norm modulo $t\Lambda$ then it has minimal norm modulo $s\Lambda$ for all $s\geq t$. We do not know if $O_{16f}$ and $O_{16g}$ have minimal norm modulo $5\Lambda$. 

In fact, we can also determine some orbits of minimal norm modulo $4\Lambda$ from the decomposition of orbit products $O_2\otimes O_2$, $O_2\otimes O_3$, $O_2\otimes O_4$, $O_3\otimes O_3$ and $O_3\otimes O_4$. These minimal-length orbits are consistent with Theorem \ref{th:system4}.

\begin{question}
We formulate some questions related to Conway invariant Jacobi forms.
\begin{enumerate}
    \item What is the minimal weight of Conway invariant weak Jacobi forms of given index? We guess that $J_{k,\Lambda,t}^{\w, \Co_0}=\{0\}$ if $k<-6t$.
    \item To find a set of $24$ algebraically independent Conway orbits. 
    \item To give a formula to describe the rank of $J_{*,\Lambda,t}^{\w,\Co_0}$, i.e. the number of $\Co_0$-orbits of $\Lambda/t\Lambda$.
    \item To give a dimensional formula for the space of Conway invariant holomorphic Jacobi forms of singular weight and given index.
    \item Is the ring $J_{*,\Lambda,*}^{\w,\Co_0}$ finitely generated? If so, what is the maximal index of the generators?
\end{enumerate}
\end{question}

\bigskip

\noindent
\textbf{Acknowledgements} 
The authors thank Daniel Allcock and Kimyeong Lee for useful discussions, and thank Brandon Williams for pointing out \cite{Mar02} and for fruitful discussions on Borcherds' thesis. The authors also thank Richard Borcherds for valuable comments. 
KS is supported by KIAS Grant (QP081001). HW thanks Max Planck Institute for Mathematics in Bonn for its
hospitality where this work was started. HW is supported by the Institute for Basic Science (IBS-R003-D1). The authors also thank the referee for many valuable comments that make this paper better. 

\appendix
\section{Product decompositions of Conway orbits}\label{app:orbitprod}
We determine the following orbit product decompositions:
\begin{equation}\nonumber
    \begin{aligned}
O_{2}\otimes O_{2}=&\,196560O_0 \oplus  4600 O_{2} \oplus  552 O_{3} \oplus  46 O_{4} \oplus  2 O_{5} \oplus  
 2 O_{6 b}
\oplus  O_{8 a}
  \\
O_{2}\otimes O_{3}=&\,47104 O_{2} \oplus  11178 O_{3} \oplus  2048 O_{4} \oplus  275 O_{5}  \oplus  
 24 O_{6 a}\oplus  O_{7} \oplus  O_{8 c}
  \\
O_{2}\otimes O_{4}=&\,93150 O_{2} \oplus  48600 O_{3} \oplus  16192 O_{4} \oplus  4050 O_{5}  \oplus  759 O_{6 a}  \oplus  891 O_{6 b}\oplus  
 100 O_{7}\\
 & \oplus  8 O_{8 b}\oplus  O_{9 a}\oplus  
 O_{10 b}
  \\
O_{3}\otimes O_{3}=&\,16773120O_0 \oplus  953856 O_{2} \oplus  257600 O_{3} \oplus  64768 O_{4} \oplus  
 14256 O_{5} \oplus  2576 O_{6 a}\\
 &\oplus  
 2816 O_{6 b}\oplus  352 O_{7}  \oplus  32 O_{8 b} \oplus  2 O_{9 b} \oplus  2 O_{10 c}\oplus  O_{12 a}
  \\
O_{2}\otimes O_{5}=&\,47104 O_{2} \oplus  75900 O_{3} \oplus  
   47104 O_{4}\oplus  19450 O_{5} \oplus  6072 O_{6a}\oplus  
   5632 O_{6b}  \oplus  
   1452 O_{7}\\
   &\oplus  
   256 O_{8b}\oplus  275 O_{8c}  \oplus  23 O_{9a} \oplus  
   33 O_{9b} \oplus  2 O_{10a} \oplus  
   O_{11a} \oplus  O_{12b}\\
O_{3}\otimes O_{4}=&\,4147200 O_{2} \oplus  1536975 O_{3} \oplus  518144 O_{4} \oplus  157950 O_{5}  \oplus  42504 O_{6 a} \oplus  
 41472 O_{6 b}\\
 &\oplus  
 9725 O_{7} \oplus  1792 O_{8 b}\oplus  
 2025 O_{8 c}\oplus  253 O_{9 a} \oplus  243 O_{9 b} \oplus  22 O_{10 a} \oplus  O_{11 b}  \oplus  O_{12 c}
   \\
O_{2}\otimes O_{6a}=&\, 48600 O_{3} \oplus  64768 O_{4} \oplus  44550 O_{5}  \oplus  
 21252 O_{6 a} \oplus  
 20736 O_{6 b}\oplus  7800 O_{7}   \oplus  2240 O_{8 b}\\
 &\oplus  
 2025 O_{8 c} \oplus  506 O_{9 a}\oplus  486 O_{9 b} \oplus  77 O_{10 a}  \oplus  100 O_{10 c} \oplus  8 O_{11 b}\oplus  O_{12 e}\oplus  O_{13 a} 
  \\
O_{2}\otimes O_{6b}=&\, 4600 O_{2} \oplus  1012 O_{4} \oplus  550 O_{5}  \oplus  276 O_{6 a}\oplus  
 1782 O_{6 b} \oplus  100 O_{7}\oplus  
 4600 O_{8 a} \oplus  28 O_{8 b}\\
 &\oplus  275 O_{8 c} \oplus  23 O_{9 a} \oplus  O_{10 a}    \oplus  44 O_{10 b} \oplus  2 O_{11 a} \oplus  3 O_{12 d}\oplus  
 O_{14 c}
 \\
O_{2}\otimes O_{8a}=&\,    O_{2} \oplus  2 O_{6 b} \oplus  O_{10 b} \oplus  O_{12 b} \oplus  
 O_{8 c} \oplus  O_{14 c}\oplus O_{18a}
\end{aligned}
\end{equation}
\section{Intersection between Conway orbits and Leech vectors}\label{app:orbitinter}
For the intersection between $O_*$ and any Leech vector of type $2$, we have
\be\nonumber
\ba
&O_2: u^{\pm 4}+4600  u^{\pm 2} +47104u^{\pm 1}+93150\\
&O_3: 512(92 u^{\pm 3}+1863 u^{\pm 2}+8100 u^{\pm1}+12650)\\
&O_4: 4050(23 u^{\pm 4}+1024 u^{\pm 3}+8096 u^{\pm 2}+23552 u^{\pm1}+32890) \\
&O_5:23552(2 u^{\pm 5}+275 u^{\pm 4}+4050 u^{\pm 3}+19450 u^{\pm 2}+45100 u^{\pm1}+58806) \\
&O_{6a}: 518400(8 u^{\pm 5}+253 u^{\pm 4}+2024 u^{\pm 3}+7176 u^{\pm 2}+14352 u^{\pm1}+17894) \\
&O_{6b}: 2300(2 u^{\pm 6}+891 u^{\pm 4}+5632 u^{\pm 3}+22518 u^{\pm 2}+41472 u^{\pm1}+55530) \\
&O_{7}: 953856(u^{\pm6}+100 u^{\pm5}+1452 u^{\pm4}+7900 u^{\pm3}+22825 u^{\pm2}+41152 u^{\pm1}+49700) \\
&O_{8a}: u^{\pm 8}+4600  u^{\pm 4} +47104u^{\pm 2}+93150 \\
&O_{8b}: 16394400(2 u^{\pm6}+64 u^{\pm5}+567 u^{\pm4}+2368 u^{\pm3}+5902 u^{\pm2}+9856 u^{\pm1}+11622) \\
&O_{8c}: 47104(u^{\pm7}+275 u^{\pm5}+2300 u^{\pm4}+9153 u^{\pm3}+24576 u^{\pm2}+37675 u^{\pm1}+48600) \\
&O_{9a}: 4147200(u^{\pm 7}+23 u^{\pm 6}+529 u^{\pm 5}+3059 u^{\pm 4}+10879 u^{\pm 3}+24035 u^{\pm 2}+37743 u^{\pm1}+44022) \\
&O_{9b}: 32972800(11 u^{\pm 6}+162 u^{\pm 5}+1053 u^{\pm 4}+3586 u^{\pm 3}+8019 u^{\pm 2}+12636 u^{\pm1}+14586)\\
&O_{10a}: 47692800(2 u^{\pm7}\!+78 u^{\pm6}\!+814 u^{\pm5}+3993 u^{\pm4}+11882 u^{\pm3}+24266 u^{\pm2}+36454 u^{\pm1}\!+41582) 
\ea
\ee
For the intersection between $O_*$ and any Leech vector of type $3$, we have
\be\nonumber
\ba
&O_2: 6(92 u^{\pm3}+1863 u^{\pm2}+8100 u^{\pm1}+12650
) \\
&O_3: u^{\pm6}+11178 u^{\pm4}+257600 u^{\pm3}+1536975 u^{\pm2}+3934656 u^{\pm1}+5292300\\
&O_4: 6075(8 u^{\pm5}+253 u^{\pm4}+2024 u^{\pm3}+7176 u^{\pm2}+14352 u^{\pm1}+17894) \\
&O_5: 276  (275 u^{\pm6}+14256 u^{\pm5}+157950 u^{\pm4}+743600 u^{\pm3}+1986525 u^{\pm2}+3434400 u^{\pm1}+4099108)\\
&O_{6a}: 16200(3 u^{\pm7}+322 u^{\pm6}+5313 u^{\pm5}+33396 u^{\pm4}+115115 u^{\pm3}+256542 u^{\pm2}+403857 u^{\pm1}\\
&\phantom{O_{6a}:}+467544) \\
&O_{6b}: 6900(11 u^{\pm6}+162 u^{\pm5}+1053 u^{\pm4}+3586 u^{\pm3}+8019 u^{\pm2}+12636 u^{\pm1}+14586
)  \\
&O_{7}: 11178(u^{\pm8}+352 u^{\pm7}+9725 u^{\pm6}+84800 u^{\pm5}+376750 u^{\pm4}+1053504 u^{\pm3}+2075603 u^{\pm2}\\
&\phantom{O_{7}: }+3053600 u^{\pm1}+3464450)\\
&O_{8a}: 6(92 u^{\pm6}+1863 u^{\pm4}+8100 u^{\pm1}+12650
) \\
&O_{8b}: 1536975(u^{\pm8}+56 u^{\pm7}+728 u^{\pm6}+4264 u^{\pm5}+14924 u^{\pm4}+36024 u^{\pm3}+64744 u^{\pm2}+90728 u^{\pm1}\\
&\phantom{O_{8b}:}+101222) \\
&O_{8c}: 552(u^{\pm9}+2025 u^{\pm7}+22528 u^{\pm6}+137700 u^{\pm5}+476928 u^{\pm4}+1151700 u^{\pm3}+2073600 u^{\pm2}\\
&\phantom{O_{8c}:}+2902878 u^{\pm1}+3238400
) 
\ea
\ee
For the intersection between $O_*$ and any Leech vector of type $4$, we have
\be\nonumber
\ba
&O_2:2(23 u^{\pm4}+1024 u^{\pm3}+8096 u^{\pm2}+23552 u^{\pm1}+32890) \\
&O_3:256(8 u^{\pm5}+253 u^{\pm4}+2024 u^{\pm3}+7176 u^{\pm2}+14352 u^{\pm1}+17894) \\
&O_4: u^{\pm8}+16192 u^{\pm6}+518144 u^{\pm5}+4595032 u^{\pm4}+19171328 u^{\pm3}+47829696 u^{\pm2}+79794176 u^{\pm1}\\
&\phantom{O_4:}+94184862 \\
&O_5:23552(2 u^{\pm7}+78 u^{\pm6}+814 u^{\pm5}+3993 u^{\pm4}+11882 u^{\pm3}+24266 u^{\pm2}+36454 u^{\pm1}+41582) \\
&O_{6a}:256( 253 u^{\pm8}+14168 u^{\pm7}+184368 u^{\pm6}+1078792 u^{\pm5}+3779498 u^{\pm4}+9114072 u^{\pm3}\\
&\phantom{O_{6a}:}+16396432 u^{\pm2}+22954184 u^{\pm 1}
+25634466 ) \\
&O_{6b}: 92(11 u^{\pm8}+512 u^{\pm7}+6864 u^{\pm6}+39936 u^{\pm5}+139854 u^{\pm4}+337920 u^{\pm3}+606832 u^{\pm2}\\
&\phantom{O_{6b}:}+850432 u^{\pm1}+949278) \\
&O_{7}: 11776(
4 u^{\pm9}+389 u^{\pm8}+6776 u^{\pm7}+48532 u^{\pm6}+200772 u^{\pm5}+564135 u^{\pm4}+1181756 u^{\pm3}\\
&\phantom{O_7:}+1943044 u^{\pm2}+2592004 u^{\pm1}+2846536) \\
&O_{8a}:2(23 u^{\pm8}+1024 u^{\pm6}+8096 u^{\pm4}+23552 u^{\pm2}+32890)  \\
&O_{8b}:4048(4 u^{\pm10}+896 u^{\pm9}+23011 u^{\pm8}+209664 u^{\pm7}+1038804 u^{\pm6}+3398784 u^{\pm5}+8194512 u^{\pm4}\\
&\phantom{O_{8b}:}+15480192 u^{\pm3}+23860008 u^{\pm2}+30652288 u^{\pm1}+33300674)  \\
&O_{8c}:47104(u^{\pm9}+22 u^{\pm8}+209 u^{\pm7}+1024 u^{\pm6}+3356 u^{\pm5}+8096 u^{\pm4}+15292 u^{\pm3}+23552 u^{\pm2}\\
&\phantom{O_{8c}:}+30294 u^{\pm1}+32868) 
\ea
\ee
\bibliographystyle{plainnat}
\bibliofont
\bibliography{refs}

\end{document}